\documentclass{amsart}
\usepackage[normalem]{ulem}
\usepackage{amssymb,nicefrac,bm,upgreek,mathtools,verbatim}
\usepackage[mathscr]{eucal}
\usepackage{dsfont}
\usepackage[final]{hyperref}
\hypersetup{
colorlinks=true,
citecolor=mblue,
linkcolor=mblue,
 urlcolor  = blue,
 anchorcolor = blue,
breaklinks}
\usepackage[msc-links]{amsrefs}

\usepackage{color}
\definecolor{dmagenta}{rgb}{.4,.1,.5}
\definecolor{dblue}{rgb}{.0,.0,.5}
\definecolor{mblue}{rgb}{.0,.0,.7}
\definecolor{ddblue}{rgb}{.0,.0,.4}
\definecolor{dred}{rgb}{.7,.0,.0}
\definecolor{dgreen}{rgb}{.0,.5,.0}
\definecolor{Eeom}{rgb}{.0,.0,.5}

\usepackage[normalem]{ulem}
\allowdisplaybreaks
\newtheorem{lemma}{Lemma}[section]
\newtheorem{theorem}{Theorem}[section]
\newtheorem{corollary}{Corollary}[section]

\theoremstyle{definition}
\newtheorem{definition}{Definition}[section]

\theoremstyle{remark}
\newtheorem{example}{Example}[section]
\newtheorem{remark}{Remark}[section]

\numberwithin{equation}{section}

\usepackage[capitalize,nameinlink]{cleveref}

\crefname{section}{Section}{Sections}
\crefname{subsection}{Subsection}{Subsections}
\crefname{condition}{Condition}{Conditions}
\crefname{hypothesis}{Hypothesis}{Conditions}
\crefname{assumption}{Assumption}{Assumptions}
\crefname{lemma}{Lemma}{Lemmas}
\crefname{claim}{Claim}{Claims}

\Crefname{figure}{Figure}{Figures}

\crefformat{equation}{\textup{#2(#1)#3}}
\crefrangeformat{equation}{\textup{#3(#1)#4--#5(#2)#6}}
\crefmultiformat{equation}{\textup{#2(#1)#3}}{ and \textup{#2(#1)#3}}
{, \textup{#2(#1)#3}}{, and \textup{#2(#1)#3}}
\crefrangemultiformat{equation}{\textup{#3(#1)#4--#5(#2)#6}}%
{ and \textup{#3(#1)#4--#5(#2)#6}}{, \textup{#3(#1)#4--#5(#2)#6}}%
{, and \textup{#3(#1)#4--#5(#2)#6}}

\Crefformat{equation}{#2Equation~\textup{(#1)}#3}
\Crefrangeformat{equation}{Equations~\textup{#3(#1)#4--#5(#2)#6}}
\Crefmultiformat{equation}{Equations~\textup{#2(#1)#3}}{ and \textup{#2(#1)#3}}
{, \textup{#2(#1)#3}}{, and \textup{#2(#1)#3}}
\Crefrangemultiformat{equation}{Equations~\textup{#3(#1)#4--#5(#2)#6}}%
{ and \textup{#3(#1)#4--#5(#2)#6}}{, \textup{#3(#1)#4--#5(#2)#6}}%
{, and \textup{#3(#1)#4--#5(#2)#6}}

\crefdefaultlabelformat{#2\textup{#1}#3}

\newcommand{\df}{\coloneqq}
\DeclareMathOperator{\Exp}{\mathbb{E}} %Expectation
\DeclareMathOperator{\Prob}{\mathbb{P}} %Probability
\newcommand{\D}{\mathrm{d}} %differential
\newcommand{\E}{\mathrm{e}} %exponent
\newcommand{\RR}{\mathbb{R}} %Real numbers
\newcommand{\Rd}{{\mathbb{R}^d}} %R^d
\newcommand{\NN}{\mathbb{N}} %Natural numbers
\newcommand{\Ind}{\mathds{1}} %indicator function
\newcommand{\Act}{\mathbb{U}} %Action Set

\newcommand{\Sob}{\mathscr{W}} %Sobolev space
\newcommand{\Sobl}{\mathscr{W}_{\mathrm{loc}}} %Sobolev space (local)
\newcommand{\Cc}{C} %Set of continuous functions
\newcommand{\Lp}{L} %Lp
\newcommand{\Lpl}{L_{\mathrm{loc}}} %Lploc

\newcommand{\lamstr}{\lambda^{\mspace{-1mu}*}} % lambda with corrected 'star'

\newcommand{\transp}{^{\mathsf{T}}} %transpose
\newcommand{\sorder}{{\mathfrak{o}}} % small Order of

\newcommand{\Lg}{\mathscr{L}} % Extended generator

\newcommand{\uuptau}{\Breve{\uptau}}
\newcommand{\grad}{\nabla}

\newcommand{\cB}{\mathcal{B}} % Borel measurable functions
\newcommand{\sB}{\mathscr{B}} % Ball used often
\newcommand{\sC}{\mathscr{C}} % Solution set of $\Lg u=0$.
\newcommand{\cC}{\mathcal{C}} % Classes of continuous functions
\newcommand{\cD}{\mathcal{D}} % Domain of \Rd
\newcommand{\sI}{\mathscr{I}}
\newcommand{\sK}{\mathscr{K}} % A compact set
\newcommand{\Lyap}{\mathcal{V}} % Lyapunov function

\newcommand{\abs}[1]{\lvert#1\rvert}
\newcommand{\norm}[1]{\lVert#1\rVert}
\newcommand{\babs}[1]{\bigl\lvert#1\bigr\rvert}

\DeclareMathOperator{\trace}{trace}

\begin{document}

\title[Liouville properties of eigenfunctions of elliptic operators]%
{Certain Liouville properties of eigenfunctions\\ of elliptic operators}

\author[Ari Arapostathis]{Ari Arapostathis}
\address{Department of ECE,
The University of Texas at Austin, 2501 Speedway, EER~7.824,
Austin, TX~~78712, USA}
\email{ari@ece.utexas.edu}
\author[Anup Biswas]{Anup Biswas}
\address{Department of Mathematics,
Indian Institute of Science Education and Research,
Dr. Homi Bhabha Road, Pune 411008, India}
\email{anup@iiserpune.ac.in}
\author[Debdip Ganguly]{Debdip Ganguly}
\address{Department of Mathematics,
Indian Institute of Science Education and Research,
Dr. Homi Bhabha Road, Pashan, Pune 411008, India}
\email{debdip@iiserpune.ac.in}
 
%%%%%%%%%%%%%%%%%%%%%%%%%%%%%%%%%%%%%%%%%%%%%%%%%%%%%%%%%%%%%%%%%%%%%%%%%%%%%%%%
\begin{abstract}
We present certain Liouville properties of eigenfunctions of
second-order elliptic operators with real coefficients,
via an approach that is based on stochastic representations of positive solutions,
and criticality theory of second-order elliptic operators.
These extend results of
Y.~Pinchover to the case of \emph{nonsymmetric} operators of
Schr\"odinger type. In particular, we provide an answer to an
open problem posed by Pinchover in
[\textit{Comm.\ Math.\ Phys.\/} \textbf{272} (2007), no.~1, 75--84, Problem~5].
In addition, we prove a lower bound on the decay of positive supersolutions of general
second-order elliptic operators in any dimension, and discuss
its implications to the Landis conjecture.
\end{abstract}



%\date{}

\maketitle

\tableofcontents

%%%%%%%%%%%%%%%%%%%%%%%%%%%%%%%%%%%%%%%%%%%%%%%%%%%%%%%%%%%%%%%%%%%%%%%%%%%%%%%%
\section{Introduction}

The main objective of this paper is to establish Liouville properties
of eigenfunctions of second-order elliptic operators.
These type of results came to prominence after the paper of 
Pinchover \cite{Pinchover-07} where he proved a very interesting property which can be
stated as follows.
Let $\cD$ be a domain in $\Rd$ and let $P_i=-\mbox{div}(Au)- V_i u$, $i=1,2$,
be two nonnegative
Schr\"{o}dinger operators, with $V_i\in \Lpl^{p}(\cD)$ for $p>\nicefrac{d}{2}$,
and $A$ locally non-degenerate in $\cD$.
Suppose that $P_1$ is critical in $\cD$ with ground state $\Psi^*_1$,
that the generalized principal eigenvalue of $P_2$ is nonnegative,
and that there exists a subsolution $\Psi$, with $\Psi^+\ne 0$,
to $P_2u=0$ in $\cD$ satisfying $\Psi^+\le C\Psi^*_1$ for some constant $C$.
Then $P_2$ is also critical with ground state $\Psi$.
In particular, the principal eigenvalue of $P_2$ equals $0$, and $\Psi>0$.
In the same paper, Pinchover proposed two problems on the generalization of 
this result for (a) general non-symmetric second
order elliptic operators and (b) quasilinear operators of $p$-Laplacian type.
Later, a similar result for $p$-Laplacian operators was proved by Pinchover, Tertikas
and Tintarev in \cite{Pinchover-08}.
However, the problem concerning general second-order elliptic operators
remains open so far.
The main goal of this paper is to address this problem for a large class of
second-order elliptic operators.

Pinchover's approach was variational.
He first established the existence of a \emph{null sequence} for the quadratic
form associated with $P_i$, and then using \emph{criticality
theory}, together with a bound on the positive part
of the subsolution, he obtained the  above mentioned Liouville-type result.
Unfortunately, for general (nonsymmetric) operators the existence of such a null sequence
is not possible, despite the fact that criticality theory is well developed for
general operators. Moreover,
in \cite[Remark~4.1]{Pinchover-07} Pinchover discussed the difficulty in
obtaining the above Liouville-type results for 
general (nonsymmetric) second-order elliptic operators.
In this paper we show that the above Liouville-type result holds for a
fairly general class of second-order elliptic operators
and potentials. 
Our approach differs significantly from variational arguments, and relies 
on stochastic representations of positive solutions studied in \cite{ABS},
and criticality theory of second-order elliptic operators. 
This allows us to bypass the use of a null sequence.

For $\cD=\Rd$, it is known that the criticality of the operator is equivalent
to the recurrence of the twisted process \cites{ABS, Pinsky}.
An interesting observation in this paper
is that criticality is also equivalent to the \emph{strict right monotonicity} of the
(generalized) principal eigenvalue. Let $\Lg$ be a second-order elliptic 
operator and $\lamstr(V)$ denote the principal eigenvalue 
of the operator $\Lg + V$, with potential $V$. We say
that $\lamstr(V)$ is \emph{strictly right monotone at} $V$, or
\emph{strictly monotone at $V$ on the right}, if
 $\lamstr(V)<\lamstr(V+h)$ for any non-zero,
nonnegative continuous function $h$ that vanishes at infinity. 
This equivalence is established in \cref{T2.1}.
We also show that given two potential functions
$V_i\in \Lpl^\infty(\Rd)$, $i=1,2$, if $V_1-V_2$ has a fixed sign outside some
compact subset of $\Rd$, then a result analogous to the one described
in the preceding paragraph
holds (see \cref{T2.2,T2.3}).
In particular, if $P_1=\Lg_1 + V_1$ is a small perturbation of $P_2=\Lg_2+V_2$,
then, under suitable assumptions, the criticality of $P_1$ implies that of $P_2$.
To further strengthen these results, we study the \emph{strict monotonicity} of the
principal eigenvalue, by which we mean that
the principal eigenvalue is strictly left and right monotone at $V$
(i.e., $\lamstr(V-h)<\lamstr(V)<\lamstr(V+h)$).
It is shown in \cref{T2.4} that if the
principal eigenvalue corresponding to $P_1$ is strictly monotone,
$\Lg_1=\Lg_2$ outside a compact subset of $\Rd$, and $V_1-V_2$ vanishes at infinity,
then under analogous hypotheses,
the principal eigenvalue of $P_2$ is also strictly monotone.
These results can be further improved to $V_i\in\Lpl^\infty(\Rd)$,
provided we impose a
`stability' assumption on $\Lg$.
See \cref{T3.1} for more details. 

The second part of this paper deals with the lower bound on the decay
of positive supersolutions of general
second-order elliptic operators in any dimension.
The results obtained here extend those of Agmon \cite{Agmon-85},
Carmona \cite{Carmona-78}, Carmona and Simon \cite{Carmona-Simon}.
Our proof is based on the stochastic representation of 
positive solutions (see \cref{T4.1}). 
As a consequence, we prove the Landis conjecture
for a large class of potentials.
Landis' conjecture \cite{Kondratev-Landis} can be loosely stated as follows:
for a bounded potential $V$, if a
solution $u$ of $\Delta u + Vu=0$ satisfies the estimate
$\abs{u(x)}\le C\exp(-c\abs{x}^{1+})$, for some positive
constants $C$ and $c$, then $u$ is identically $0$.
For a precise statement, we refer the reader to \cref{Sec-4}.
This conjecture is open when $u$ and $V$ are real valued, while a counterexample
was constructed by Meshkov \cite{Meshkov} for complex-valued $u$ and $V$.
This conjecture was revisited recently in
\cites{Davey-Kenig-Wang,Kenig-Silvestre-Wang}, for dimension $2$ and for $V\le 0$.
Note that this conjecture trivially holds for $V\le 0$ due to
the strong maximum principle.
The main contribution in \cites{Davey-Kenig-Wang,Kenig-Silvestre-Wang} is the lower
bound on the decay rate of solutions that may not vanish at infinity.
In this direction, Kenig has conjectured in \cite[Question~1]{Kenig-05} a
lower bound on the decay of the eigenfunctions of the Schr\"{o}dinger's equation.
In \cref{T4.3} we validate the Landis conjecture for a large class
of potentials and in any dimension $d\ge 2$.
This class of potentials includes compactly supported functions.
We also wish to bring to the attention of the reader a recent study
of Liouville properties for nonlinear operators \cite{Bardi-16}.

The paper is organized as follows.
In \cref{Sec-2} we briefly review some basic results
from the criticality theory of second-order elliptic operators
and state our main results.
\cref{Sec-3} is devoted to the proofs of \cref{T2.1,,T2.2,T2.3,T2.4}.
In \cref{Sec-4} we establish a lower bound on the decay of positive
supersolutions (\cref{T4.1}), and discuss its implication
to Landis conjecture. 

\subsection*{Notation}
The open ball of radius $r$ around a point $x\in\Rd$ is denoted
by $B_r(x)$, and $B_r$ stands for $B_r(0)$.
By $\cC_0(\Rd)$ ($\cB_0(\Rd)$) we denote the collection of all
real valued continuous (Borel measurable) functions on
$\Rd$ that vanish at infinity.
By $\norm{\cdot}_\infty$ we denote the $L^\infty$ norm.
Also $\kappa_1, \kappa_2, \dotsc$ are used as generic constants whose values
might vary from place to place.

%%%%%%%%%%%%%%%%%%%%%%%%%%%%%%%%%%%%%%%%%%%%%%%%%%%%%%%%%%%%%%%%%%%%%%%%%%%%%%%%
\section{Preliminaries and main results}\label{Sec-2}
In this section we introduce our assumptions and state our main results.
The conditions (A1)--(A3) on the coefficients of the operator that follow
are used in most of the results of the paper, so we assume that
they are in effect throughout unless otherwise mentioned.
A notable exception to this is \cref{T2.2}, where only (A3) is assumed.
\begin{enumerate}
\item[(A1)]\label{A1}
\emph{Local Lipschitz continuity:\/}
The function
$a=\bigl[a^{ij}\bigr]\,\colon\,\RR^{d}\to\mathcal{S}^{d\times d}_+$, 
where $\mathcal{S}^{d\times d}_+$ denotes the set of real, 
symmetric positive definite matrices,
is locally Lipschitz in $x$ with a Lipschitz constant $C_{R}>0$
depending on $R>0$.
In other words, we have
\begin{equation*}
\norm{a(x) - a(y)}
\,\le\,C_{R}\,\abs{x-y}\qquad\forall\,x,y\in B_R\,,
\end{equation*}
where $\norm{a}^{2}\df\mathrm{trace}\left(aa\transp\right)$.
The drift function $b\,\colon\Rd\to\RR$ is a locally bounded Borel measurable function.

\item[(A2)]
\emph{Affine growth condition:\/}
$b$ and $a$ satisfy a global growth condition of the form
\begin{equation*}
\langle b(x),x\rangle^+ + \norm{a(x)}\,\le\,C_0
\bigl(1 + \abs{x}^{2}\bigr) \qquad \forall\, x\in\RR^{d},
\end{equation*}
for some constant $C_0>0$.

\item[(A3)]
\emph{Nondegeneracy:\/}
For each $R>0$, it holds that
\begin{equation*}
\sum_{i,j=1}^{d} a^{ij}(x)\xi_{i}\xi_{j}
\,\ge\,C^{-1}_{R} \abs{\xi}^{2} \qquad\forall\, x\in B_{R}\,,
\end{equation*}
and for all $\xi=(\xi_{1},\dotsc,\xi_{d})\transp\in\RR^{d}$.

\end{enumerate}

We define $\upsigma(x)= \sqrt{2} a^{\nicefrac{1}{2}}(x)$.
Then under (A1) and (A3),
$\upsigma$ is also locally Lipschitz and has at most linear growth.
We say that $a$ is uniformly
elliptic if (A3) holds for a positive $C=C_R$ which is independent of $R$.

Consider the It\^{o} stochastic differential equation (SDE) given by
\begin{equation}\label{E2.1}
\D{X_s} \,=\, b(X_s)\, \D{s} + \upsigma(X_s)\, \D{W_s}\,,
\end{equation}
where $W$ is a standard $d$-dimensional Wiener process
defined on some complete,
filtered probability space $(\Omega, \mathfrak{F}, \{\mathfrak{F}_t\}, \Prob)$.
By a strong solution of \cref{E2.1} we mean an $\mathfrak{F}_t$--adapted
process $X_t$ which satisfies
\begin{equation*}
X_t \,=\, X_0 + \int_0^t b(X_s)\, \D{s} + \int_0^t \upsigma(X_s)\, \D{W_s}\,,
\quad t\ge 0, \quad \text{a.s.}\,,
\end{equation*}
where third term on the right hand side is an It\^{o} stochastic integral.
It is well known that given a complete, filtered probability space
$(\Omega, \mathfrak{F}, \{\mathfrak{F}_t\}, \Prob)$
with a Wiener process $W$, 
there exists a unique strong solution of \cref{E2.1}
\cite[Theorem~2.8]{Gyongy-96}.
The process $X$ is also strong Markov,
and we denote its transition kernel by $P^{t}(x,\cdot\,)$.
It also follows from the work in \cite{Bogachev-01}
that the transition probabilities of $X$
have densities which are locally H\"older continuous.
The \emph{extended generator} $\Lg$ is given by
\begin{equation}\label{E-Lg}
\Lg f(x) \,=\, a^{ij}(x)\,\partial_{ij} f(x)
+ b^{i}(x)\, \partial_{i} f(x)\,,
\end{equation}
for $f\in\Cc^{2}(\RR^{d})$.
The operator $\Lg$
is the generator of a strongly-continuous
semigroup on $\Cc_{b}(\RR^{d})$, which is strong Feller.
We let $\Prob_{x}$ denote the probability measure, and
$\Exp_{x}$ the expectation operator on the canonical space of the
process conditioned on $X_0=x$.

The closure, boundary, and the complement
of a set $A\subset\Rd$ are denoted
by $\Bar{A}$, $\partial{A}$, and $A^{c}$, respectively.
We write $A\Subset B$ to indicate that $\Bar{A}\subset B$.
By $\uptau(D)$
we denote the first exit time of the process $X$ from a domain
$D\subset\Rd$, i.e.,
\begin{equation*}
\uptau(D)\,\df\,\inf\; \{t\, \colon X_{t}\notin D\}\,.
\end{equation*}
The process $X$ is said to be \emph{recurrent} if for any
bounded domain $D$ we have $\Prob_x(\uptau(D^c) < \infty) = 1$
for all $x\in \Bar{D}^c$. Otherwise the process is called \emph{transient}.
A recurrent process is said to be \emph{positive recurrent} if
$\Exp_x[\uptau(D^c)] < \infty$ for all $x\in \Bar{D}^c$.
It is known that for a non-degenerate diffusion the property of recurrence
(or positive recurrence) is independent of domain $D$ and $x$, i.e., if it holds
for some domain $D$ and some $x\in \Bar{D}^c$, then it also holds for every
bounded domain $D$,
and all $x\in \Bar{D}^c$ \cite[Theorem~2.6.12 and Theorem~2.6.10]{book}.
By $\uuptau_r$ ($\uptau_r$) we denote the first hitting (exit) time of the
ball $B_r$ of radius $r$ around
$0$, i.e., $\uuptau_r=\uptau(B_r^c)$ and $\uptau_r=\uptau(B_r)$.

%%%%%%%%%%%%%%%%%%%%%%%%%%%%%%%%%%%%%%%%%%%%%%%%%%%%%%%%%%%%%%%%%%%%%%%%%%%%%%%%
In order to state the results in this paper, we review some basic definitions from 
criticality theory which have been introduced by various authors 
\cites{Agmon-83, Barry-81, Murata-86, MMU},
and have been further developed by Y.~Pinchover
(see \cites{Pinchover-88, Pinchover-89, Pinchover-90} and references therein). 
The reader should keep in mind that although the convention in criticality
theory is to consider the eigenvalues of the operator $-\Lg$,
we find it more convenient to work with the eigenvalues of $\Lg$.

\begin{definition}
Throughout the paper, $\cD \subset \Rd$ denotes a domain,
and $\mathfrak{D}\df\{ \cD_j \}_{j = 1}^{\infty}$ a sequence of
bounded subdomains with smooth boundaries, such that
$\Bar \cD_j \subset \cD_{j +1}$, and $\cD = \cup_{j = 1}^{\infty} \cD_j$.
We denote the cone of all positive solutions of the
equation $\Lg u = 0$ in $\cD$ by $\sC_{\Lg}(\cD)$.
We always assume that solutions
$u$ are in $\Sobl^{2, d}(\cD)$, i.e.,
$u$ is a strong solution, so that $\Lg u$ is defined pointwise almost everywhere.

Given a \emph{potential} $V\in L^\infty_{\mathrm{loc}}(\cD)$,
we introduce the operator
\begin{equation*}
\Lg_{V} \,\df\, \Lg + V\,.
\end{equation*}
We say that $-\Lg_V$ is \emph{nonnegative in} $\cD$
(and denote it by $-\Lg_V\ge 0$ in $\cD$),
if $\sC_{\Lg_V}(\cD)\ne \varnothing$.
The generalized principal eigenvalue of the operator $\Lg_{V}$ is defined by
\begin{equation*}
\lamstr(\Lg, V) \,\df\, \inf\; \bigl\{ \lambda \in \mathbb{R}\,\colon
\sC_{\Lg_{V} - \lambda}(\cD) \ne\varnothing \bigr\}\,.
\end{equation*}
Note that $-\Lg_{V}$ is nonnegative in $\cD$ if and only if
$ \lamstr(\Lg,V)\le 0$. 
\end{definition}

It is clear that $-\Lg$ is always nonnegative, since
$\bm1\in\sC_{\Lg}(\cD)$,
where $\bm1$ is the constant function on $\cD$ having value $1$
at every $x\in \cD$.
In the sequel we shall use the notation $\lamstr(V)$ instead of $\lamstr(\Lg, V)$,
whenever this is not ambiguous.
In most of the paper we deal with the case $\cD=\Rd$.
An exception to this is \cref{T2.2},
where we address the question of Pinchover \cite[Problem~5]{Pinchover-07}
for general domains $\cD$.

Let us now recall the definitions of critical and subcritical operators
and the ground state. 

%%%%%%%%%%%%%%%%%%%%%%%%%%%%%%%%%%%%%%%%%%%%%%%%%%%%%%%%%%%%%%%%%%%%%%%%%%%%%%%%
\begin{definition}[Minimal growth at infinity]\label{MGI}
A positive function $u\in \Sobl^{2, d}(\cD)$ satisfying 
\begin{equation*}
\Lg_{V} u \,=\, 0 \quad \text{a.e. in}\ \cD\,,
\end{equation*}
is said to be a solution of minimal growth at infinity,
if for any compact $K\subset \cD$ and any positive function
$v\in \Sobl^{2, d}(\cD\setminus K)$ which 
satisfies $\Lg_{V} v\le 0$ a.e. in $\cD\setminus K$,
there exist $\cD_i\in\mathfrak{D}$, with $K\subset\cD_i$, and a constant $\kappa>0$
such that $\kappa u\le v$ in $\Bar\cD_i^c\cap \cD$.
A positive solution $u \in \sC_{\Lg_{V}}(\cD)$
which has minimal growth at infinity in
$\cD$ is called the (\emph{Agmon}) \emph{ground state} of $\Lg_{V}$ in $\cD$.
\end{definition}

%%%%%%%%%%%%%%%%%%%%%%%%%%%%%%%%%%%%%%%%%%%%%%%%%%%%%%%%%%%%%%%%%%%%%%%%%%%%%%%%
\begin{remark}%\label{R2.1}
\Cref{MGI} is equivalent to what
is generally used in criticality theory.
In criticality theory for an operator $P$, a function $u \in \sC_{P}(\cD)$ is
said to have a minimal growth at infinity in $\cD$, if for
any $K\Subset \cD$, with a smooth boundary,
and any positive supersolution 
$v\in\Sobl^{2, d}(\cD\setminus K)$ of $Pv=0$ in 
$\cD \setminus K$ such that 
$v\in \cC((\cD \setminus K)\cup \partial K)$, and
$u \le v$ on $\partial K$, it holds that $u \le v$ in $\cD \setminus K$.

It is easy to see that this definition implies minimal growth at infinity
according to
\cref{MGI} for $P=\Lg_{V}$.
To see the converse direction, define
\begin{equation*}
\kappa_0 \,=\, \inf_{\cD\cap K^c}\;\frac{v}{u}\,\,.
\end{equation*}
Since $u\le v$ on $\partial K$, we must have $\kappa_0\le 1$ by continuity.
We claim that $\kappa_0=1$.
Arguing by contradiction, suppose that $\kappa_0<1$.
Then $v-\kappa_0 u$ must be positive on $\cD\setminus K$ by the strong
maximum principle.
Since $P(v-\kappa_0 u)\le 0$, then
by \cref{MGI} there exist $\kappa\in(0, 1- \kappa_0)$
and $\cD_i\in\mathfrak{D}$,
such that $\kappa u\le v-\kappa_0 u$ in $\Bar\cD_i^c\cap \cD$.
Without loss of generality, suppose that $\cD_i\supset K$.
Applying the strong maximum principle in $\cD_i\cap K^c$ to 
\begin{equation*}
\Lg \Phi - V^-\Phi\,\le\, 0\,, \quad \Phi \,=\, v-(\kappa_0+\kappa) u\,,
\end{equation*}
we have $\kappa u\le v-\kappa_0 u$ in $\Bar\cD_i\cap K^c$,
and therefore $(\kappa_0+\kappa) u\le v$ in $\cD\cap K^c$.
But this contradicts the
definition of $\kappa_0$. Hence $\kappa_0=1$.
\end{remark}

%%%%%%%%%%%%%%%%%%%%%%%%%%%%%%%%%%%%%%%%%%%%%%%%%%%%%%%%%%%%%%%%%%%%%%%%%%%%%%%%
\begin{definition}%\label{D2.3}
The operator $\Lg_{V}$ is said to be \emph{critical} in $\cD$, if
$\Lg_{V}$ admits a ground state in $\cD$. 
The operator $\Lg_{V}$ is called \emph{subcritical} in $\cD$, if 
$-\Lg_{V}\ge 0$ in $\cD$, but $\Lg_{V}$ does not admit
a ground state solution.
\end{definition}

%%%%%%%%%%%%%%%%%%%%%%%%%%%%%%%%%%%%%%%%%%%%%%%%%%%%%%%%%%%%%%%%%%%%%%%%%%%%%%%%
\begin{example}
Let $\Lg = \Delta$ in $\mathbb{R}^d$, $d \ge 1$.
It is well known that $\lambda^{*}(\Lg,0) = 0$.
Moreover, $\Lg$ is critical if and only if $d \le 2$.
\end{example}

%%%%%%%%%%%%%%%%%%%%%%%%%%%%%%%%%%%%%%%%%%%%%%%%%%%%%%%%%%%%%%%%%%%%%%%%%%%%%%%%
\begin{example}
Let $\cD = \mathbb{R}^d \setminus \{ 0 \}$, $d \ge 3$, and consider the Hardy operator 
\begin{equation*}
\Lg_{V} \df \Delta + \frac{(d-2)^2}{4} \frac{1}{|x|^2}\,.
\end{equation*}
Then it is well known that $\Lg_V$ is critical,
and the corresponding ground state is $|x|^{\frac{2-d}{2}}$.
\end{example}

%%%%%%%%%%%%%%%%%%%%%%%%%%%%%%%%%%%%%%%%%%%%%%%%%%%%%%%%%%%%%%%%%%%%%%%%%%%%%%%%
\begin{remark}
For $\cD=\Rd$ one can also define the (generalized) principal eigenvalue in the sense of
Berestycki and Rossi \cite{Berestycki-15} (see also \cite{NP-92}) by
\begin{equation*}
\Hat\lambda^*(\Lg, V)\,\df\, \inf\;\bigl\{\lambda\in\RR\,\colon\, \exists\,
\varphi\in\Sobl^{2, d}(\Rd),~ \varphi>0,~ \Lg_{V} \varphi - \lambda \varphi\le 0
\text{\ a.e. in\ }\Rd\bigr\}\,\,.
\end{equation*}
For $V\in\Lpl^\infty(\Rd)$, it is known from \cite[Theorem~1.4]{Berestycki-15}
that there exists a (generalized) positive eigenfunction corresponding to
$\Hat\lambda^*(\Lg, V)$, whenever this is finite.
Thus $\Hat\lambda^*(\Lg, V)=\lambda^*(\Lg, V)$.
\end{remark}

%%%%%%%%%%%%%%%%%%%%%%%%%%%%%%%%%%%%%%%%%%%%%%%%%%%%%%%%%%%%%%%%%%%%%%%%%%%%%%%%
\begin{remark}%\label{R2.3}
Let $P=\Lg_{V}$ in $\cD$.
It is well known that the operator $P$ is critical in $\cD$, if and only if
the equation $P u = 0$ in $\cD$ has a unique (up to a multiplicative constant)
positive supersolution (see \cites{Pinchover-88, Pinchover-89}).
In particular, $P$ is critical in $\cD$ if and only if $P$ does not admit a
positive Green's function in $\cD$.
However, there exists a sign-changing Green's function for a $P$
which is critical in $\cD$
(see \cite{Gan-Pinch-16}).
In addition, in the critical case, we have $\dim \sC_{P}(\cD) = 1$,
and the unique positive solution (up to a multiplicative positive constant)
is a ground state of $P$ in $\cD$. 

On the other hand, $P$ is subcritical in $\cD$ if and only if $P$ admits a unique
positive minimal Green's function $G_{P}^{\cD}(x,y)$ in $\cD$.
Moreover, for any fixed $y\in \cD$, the function $G_{P}^{\cD}(\cdot,y)$ is a
positive solution of minimal growth in a neighborhood of infinity in
$\cD$, i.e., in $\cD\setminus K$ for some compact set $K$
(see \cite{BMY}). 
\end{remark}

%%%%%%%%%%%%%%%%%%%%%%%%%%%%%%%%%%%%%%%%%%%%%%%%%%%%%%%%%%%%%%%%%%%%%%%%%%%%%%%%

For an eigenpair $(\Psi, \lambda)$ of $\Lg_{V}$ in $\Rd$, i.e.,
a solution of
\begin{equation*}
\Lg_{V} \Psi \,=\, \lambda \Psi, \quad \Psi>0\quad \text{in\ } \Rd\,,
\end{equation*}
the \emph{twisted process corresponding to $(\Psi, \lambda)$} is defined by the SDE
\begin{equation}\label{twisted}
\D{Y}_s \,=\, b(Y_s)\,\D{s} + 2a(Y_s)\grad\psi(Y_s)\, \D{s} + \upsigma(Y_s)\,\D{W_s}\,,
\end{equation}
with $\psi=\log\Psi$.
The process $Y$ also goes by the name of
Doob's h-transformation in the literature.
Since $\psi\in \Sobl^{2, p}(\Rd)$, $p > d$,
it follows that $\psi$ is locally bounded (in fact, it is locally H\"{o}lder continuous), 
and therefore \cref{twisted} has a unique strong solution up to its explosion time.
In what follows, we use the notation $(\Psi^*, \lamstr (V))$ to denote a principal
eigenpair.

Let us introduce one more definition which is related to the criticality
of an operator. By $\cC^+_0(\Rd)$ we denote
the collection of all nonnegative, non-zero, real valued continuous functions on
$\Rd$ that vanish at infinity.
We fix $\Lg$, and dropping the dependence on $\Lg$ in the notation,
as mentioned earlier,
we let $\lamstr(V)$ denote the principal eigenvalue of $\Lg+V$.

%%%%%%%%%%%%%%%%%%%%%%%%%%%%%%%%%%%%%%%%%%%%%%%%%%%%%%%%%%%%%%%%%%%%%%%%%%%%%%%%
\begin{definition}
$\lamstr(V)$ is said to be \emph{strictly monotone at $V$} if for all
$h\in \cC^+_0(\Rd)$ we have $\lamstr(V-h)<\lamstr(V)<\lamstr(V+h)$.
Also, $\lamstr(V)$ is said to be \emph{strictly monotone at $V$ on the right}
if for all $h\in \cC^+_0(\Rd)$ we have
$\lamstr(V)<\lamstr(V+h)$.
\end{definition}

It is known from \cite[Theorem~2.2]{ABS} that if $\lamstr(V-h)<\lamstr(V)$
for some $h\in\cC^+_0(\Rd)$, then $\lamstr(V-h)<\lamstr(V)<\lamstr(V+h)$
for all $h\in\cC^+_0(\Rd)$.
This assertion also follows using the fact that
$V\mapsto \lamstr(V)$ is a convex function
(see for instance, \cite{Berestycki-15}).

%%%%%%%%%%%%%%%%%%%%%%%%%%%%%%%%%%%%%%%%%%%%%%%%%%%%%%%%%%%%%%%%%%%%%%%%%%%%%%%%
\begin{example}
For $\Lg=\Delta$ in $\mathbb{R}^2$ and $V=0$,
it is known that $\lamstr(V)$ strictly monotone at $V$ on the right,
 but not strictly monotone at $V$. 
\end{example}

Throughout the paper, with the exception of \cref{T2.2},
we consider potential functions $V$
that are Borel measurable and bounded from below.
We also assume that $\lamstr(V)$ is finite.
Let us begin with the following equivalence between the strict right monotonicity
of the principal eigenvalue and the criticality of the operator \cite{ABS}.
See also \cite[Theorems~4.3.3 and 7.3.6]{Pinsky} for similar results
concerning operators with regular coefficients.

%%%%%%%%%%%%%%%%%%%%%%%%%%%%%%%%%%%%%%%%%%%%%%%%%%%%%%%%%%%%%%%%%%%%%%%%%%%%%%%%
\begin{theorem}\label{T2.1}
Let $\cD=\Rd$. The following are equivalent.
\begin{itemize}
\item[(a)]
A function $\Psi\in \Sobl^{2,d}(\cD)$ is a ground state for
$\Lg_V-\lambda$, with $\lambda\in\RR$.

\item[(b)]
The twisted process corresponding to the eigenpair $(\Psi,\lambda)$ is recurrent.

\item[(c)]
$\lamstr (V)$ is strictly monotone at $V$ on the right.

\item[(d)]
For any $r>0$, the eigenpair $(\Psi,\lambda)$ satisfies
\begin{equation}\label{ET2.1A}
\Psi(x)\,=\, \Exp_x\Bigl[e^{\int_0^{\uuptau_r}(V(X_s)-\lambda)\, \D{s}}\,
\Psi(X_{\uuptau_r}) \Ind_{\{\uuptau_r<\infty\}}\Bigr]\,,
\quad x\in B^c_r\,,
\end{equation}
where, as defined earlier,
$\uuptau_r$ denotes the first hitting time to the ball $B_r$.
\end{itemize}
\end{theorem}

We often exploit the above equivalence
between strict monotonicity and criticality.
To state our next result we need some additional notation. Let 
\begin{equation*}
\Lg_k f \,=\, a^{ij}_k(x)\,\partial_{ij} f(x)
+ b^{i}_k(x)\, \partial_{i} f(x)\,,\quad k=1,2\,.
\end{equation*}
We assume that $(a_k, b_k)$, $k=1,2$, satisfies (A1)--(A3).
We say that $\Lg_1$ is a \emph{small perturbation} \cite{Pinchover-88}
of $\Lg_2$ if $\norm{a_1(x)-a_2(x)} + \abs{b_1(x)-b_2(x)}=0$ outside some compact set.
The first main result of this section is the following theorem
which gives a partial answer (see also \cref{T2.3}) to the open question posed
by Y.~Pinchover in
\cite[Problem~5]{Pinchover-07}.
Simplifying the notation, in the sequel we sometimes denote
by $\lamstr_k$ (instead of $\lamstr(\Lg_k, V_k)$) the principal eigenvalue
of the operator $\Lg_k + V_k$, $k = 1,2$.

%%%%%%%%%%%%%%%%%%%%%%%%%%%%%%%%%%%%%%%%%%%%%%%%%%%%%%%%%%%%%%%%%%%%%%%%%%%%%%%%
\begin{theorem}\label{T2.2}
Let $\cD$ be a domain in $\mathbb{R}^d$, $d \ge 1$.
Consider two Schr\"odinger operators defined on $\cD$ of the form 
\begin{equation*}
P_{k} \,\df\, \Lg_{k} + V_k\,, \quad k = 1, 2, 
\end{equation*}
where $a_k$, $k=1,2$, are continuous and satisfy (A3),
$b_k, V_k \in \Lpl^{\infty}(\cD)$,
and $V_2 \ge V_1$ outside a compact set in $\cD$.
In addition, assume that $\Lg_1$ is a small perturbation of $\Lg_2$ in $\cD$, and 
\begin{enumerate}
\item
The operator $P_1-\lamstr_1$ is critical in $\cD$.
Denote by $\Psi_1^{*}$ its ground state. 

\item
$\lamstr_2 \le \lamstr_1$ and there exists $\Psi\in\Sobl^{2, d}(\cD)$,
with $\Psi^+\ne 0$, satisfying
\begin{equation}\label{ET2.2A}
\Lg_2\Psi + V_2\Psi \,\ge\, \lamstr_1\Psi\,,
\end{equation}
and
\begin{equation}\label{ET2.2B}
\Psi^+(x)\,\le\, C\, \Psi^{*}_1(x)\quad \text{for all\ } x\in\cD\,,
\end{equation}
for some constant $C>0$.
\end{enumerate}
Then the operator $P_2-\lamstr_2$ is critical in $\cD$,
$\lamstr_1=\lamstr_2$, and $\Psi$ is its ground state. 
\end{theorem}

%%%%%%%%%%%%%%%%%%%%%%%%%%%%%%%%%%%%%%%%%%%%%%%%%%%%%%%%%%%%%%%%%%%%%%%%%%%%%%%%
\begin{remark}
One can not expect any pair $V_1$, $V_2$ to satisfy the hypotheses of
\cref{T2.2}, even if we restrict $V_1$ and $V_2$ to have compact support, and
consider the same second-order operator $\Lg=\Lg_1=\Lg_2$.
To see this, let us take $V_2\lneqq V_1$, both of them compactly supported,
and suppose that $\Lg+V_1-\lamstr_1$ is critical.
Then it is not possible to have $\lamstr_1=\lamstr_2$ and $\Psi^*_2\le C \Psi^*_1$,
for some constant $C$.
Indeed, if $\Tilde\Lg$ denotes the generator of the twisted process corresponding
to the eigenpair $(\Psi^*_1, \lamstr_1)$, and
$\lamstr_1=\lamstr_2$, then for $\Phi=\frac{\Psi^*_2}{\Psi^*_1}$ we have
\begin{equation*}
\Tilde\Lg \Phi \,=\, (V_1-V_2)\Phi \,\ge\, 0\, \,.
\end{equation*}
Since the twisted process $Y_s$ corresponding to $(\Psi^*_1, \lamstr_1)$ is recurrent
by \cref{T2.1},
and $\Phi(Y_s)$ is a bounded submartingale, $\Phi$ must be constant.
This implies $(V_1-V_2)\Psi^*_1=0$, which is a contradiction.
 
More precisely, one can find a relation between $V_1$ and $V_2$ as follows.
Suppose $\cD=\Rd$ and the operators $\Lg+ V_i$, $V_i\in\Lpl^{\infty}(\Rd)$,
are critical in $\Rd$ with principal eigenfunctions
$\Psi^*_i$, $i=1,2$. Then by \cref{T2.1} we know that for any $r>0$ we have
\begin{equation}\label{ER2.4A}
\Psi^*_i(x)\,=\, \Exp_x\Bigl[e^{\int_0^{\uuptau_r}V_i(X_s)\, \D{s}}\,
\Psi^*_i(X_{\uuptau_r}) \Ind_{\{\uuptau_r<\infty\}}\Bigr]\,,
\quad x\in B^c_r\,.
\end{equation}
Now if \cref{ET2.2B} holds, i.e., $\Psi^*_2\le C\Psi^*_1$ in $\Rd$,
then by \cref{ER2.4A}, for every $r>0$ we can find a constant $C_r$ such that
\begin{equation*}
\Exp_x\Bigl[e^{\int_0^{\uuptau_r}V_2(X_s)\, \D{s}}\, 
\Ind_{\{\uuptau_r<\infty\}}\Bigr]\,\le\, C_r
\Exp_x\Bigl[e^{\int_0^{\uuptau_r}V_1(X_s)\, \D{s}}\,
\Ind_{\{\uuptau_r<\infty\}}\Bigr]\,,\quad x\in B^c_r\,.
\end{equation*}
This in particular, provides a necessary condition on the potentials
for Liouville type theorems like \cref{T2.2} to hold.
\end{remark}

For the rest of the results in this section we let $\cD=\Rd$.

%%%%%%%%%%%%%%%%%%%%%%%%%%%%%%%%%%%%%%%%%%%%%%%%%%%%%%%%%%%%%%%%%%%%%%%%%%%%%%%%
\begin{theorem}\label{T2.3}
Consider two Schr\"odinger operators defined on $\Rd$ of the form 
\begin{equation*}
P_{k} \,\df\, \Lg_k + V_k\,, \quad k = 1, 2, 
\end{equation*}
whose coefficients satisfy (A1)--(A3), and
$V_k \in \Lpl^\infty (\Rd)$.
Let
\begin{equation*}
\Tilde{V}(x)\,\df\,\max\{V_1(x), V_2(x)\}\,.
\end{equation*}
Suppose that there exists a positive $\Tilde\Phi\in\Sobl^{2, d}(\Rd)$ and
a compact set $K$ such that
\begin{equation}\label{ET2.3A}
\Lg_1= \Lg_2, \quad \Lg_2 \Tilde\Phi + \Tilde{V}\Tilde\Phi \,\le\,
\lamstr_1 \Tilde\Phi \quad \text{in\ } K^c\,.
\end{equation}
In addition, assume that 
\begin{enumerate}
\item
The operator $P_1-\lamstr_1$ is critical in $\Rd$.
Denote by $\Psi_1^{*}$ its ground state. 
\item
$\lamstr_2 \le \lamstr_1$, and there exists subsolution $\Psi\in\Sobl^{2, d}(\Rd)$,
which may be sign-changing but $\Psi^+\ne 0$, that satisfies
\begin{equation}\label{ET2.3B}
\Lg_2\Psi + V_2\Psi \,\ge\, \lamstr_1\Psi\,,
\end{equation}
and for some constant $C>0$, 
\begin{equation*}
\Psi^+(x)\,\le\, C\, \Psi^{*}_1(x)\quad \text{for all\ } x\in\Rd\,.
\end{equation*}
\end{enumerate}
Then the operator $P_2-\lamstr_2$ is critical in $\Rd$,
$\lamstr_1=\lamstr_2$, and $\Psi$ is its ground state. 
\end{theorem}

It should be noted that the second display in \cref{ET2.3A} is an assumption on
the operators; compare to \cite[Theorem~1.7]{Pinchover-07}.
However, there is a large family of elliptic operators for which \cref{ET2.3A} holds,
as the following examples show.

%%%%%%%%%%%%%%%%%%%%%%%%%%%%%%%%%%%%%%%%%%%%%%%%%%%%%%%%%%%%%%%%%%%%%%%%%%%%%%%%
\begin{example}\label{EG2.4}
Note that \cref{ET2.3A} is satisfied if $V_1-V_2$ has a fixed sign outside
a compact set $K$. If $\Tilde{V}=V_1$ in $K^c$ we can choose $\Tilde\Phi=\Psi^*_1$.
On the other hand, if $\Tilde{V}=V_2$ in $K^c$, we know from
\cite[Theorem~1.4]{Berestycki-15} that there exists a positive $\Phi_1$ satisfying
\begin{equation*}
\Lg_2\Phi_1 + V_2\Phi_1\,=\, \lamstr_1\Phi_1 \quad \text{in\ } \Rd\,.
\end{equation*}
Hence we can take $\Tilde\Phi=\Phi_1$ in $K^c$.
\end{example}

%%%%%%%%%%%%%%%%%%%%%%%%%%%%%%%%%%%%%%%%%%%%%%%%%%%%%%%%%%%%%%%%%%%%%%%%%%%%%%%%
\begin{example}\label{EG2.5}
Let us now give an example where the sign of $V_1 - V_2$ may not be fixed
outside some compact set.
Consider $P_1 = \Delta + V_1$ in $\mathbb{R}^d$, $d \ge 3$,
such that $V_1$ has compact support and $P_1$ is critical in $\mathbb{R}^d$ with
$\lamstr=0$.
Now let 
\begin{equation*}
\Lg_2\,\df\, \Delta + b^i \partial_i\,,
\end{equation*}
where the vector field $b$ has compact support in $\Rd$. 
Let a nonnegative $\tilde{W}$
be a \emph{small perturbation}
(see \cite[Definition~2.1, Example~2.2]{Pinchover-89})
with respect to the operator $-\Delta$.
Then there exists a positive $\varepsilon$
such that $-\Delta \Psi-\varepsilon \tilde{W} \Psi=0$ has a positive solution
$\Psi$ in $B_r^c$,
$r>0$ \cite[Lemma~2.4]{Pinchover-89}.
Therefore, if we choose a potential
$V_2$ which decays faster than $\Tilde{W}$ at infinity, i.e.,
for every $\delta>0$ there exists a compact $K_\delta$ such that
$\abs{V_2(x)}\le \delta \tilde{W}(x)$ for $x\in K^c_\delta$, it is easy to see that,
by choosing $\delta<\varepsilon$, we have
\begin{equation*}
\Lg_2 \Psi + \Tilde{V}\Psi\,\le\, \Delta \Psi + \abs{V_2}\Psi\;\le \; 0
\quad \text{on the complement of a compact set in\ } \Rd\,.
\end{equation*}
Therefore \cref{ET2.3A} holds. 

There are several choices for a small perturbation $\tilde{W}$
(see \cite[Example~2.2]{Pinchover-89}).
For instance, we could take any nonnegative $\tilde{W}$ which is locally H\"{o}lder
continuous and satisfies
\begin{equation*}
(1+\abs{x})^2\tilde{W}(x)\le \varphi(\abs{x}) \quad \forall \; x\in\Rd\,,
\quad \text{and}
\quad \int_{r_0}^\infty \frac{1}{r}\varphi(r)\,\D{r}\;<\; \infty\,, \quad r_0>0\,.
\end{equation*}
\end{example}

%%%%%%%%%%%%%%%%%%%%%%%%%%%%%%%%%%%%%%%%%%%%%%%%%%%%%%%%%%%%%%%%%%%%%%%%%%%%%%%%
\begin{example}\label{EG2.6}
We define for $k = 1, 2$,
\begin{equation*}
P_k \,\df\, \Delta + b^i_k \partial_i + V_k \quad \text{in\ } \Rd\,,\ d \ge 3\,,
\end{equation*}
with the vector fields $b_k$ smooth, and satisfying 
\begin{equation*}
\abs{b_k(x)} \,\le\, \frac{C}{(1 + |x|)^{1 + \varepsilon}}\,,
\end{equation*}
for some constants $C > 0$ and $\varepsilon >0$. It is known that the operator
$\Lg_k \df \Delta + b^i_k \partial_i $ is subcritical;
this follows from
the fact that $\bf{1}$ is a positive solution,
together with the above decay estimate on $b_k$.
Hence there exists a minimal Green's function for $\Lg_k$.
Also, the Green functions $G_{-\Delta}$ and $G_{-\Lg_k}$
are comparable \cite[Theorem~1]{Ancona-97}.
Therefore, a small perturbation of $\Delta$ is also a
small perturbation of $\Lg_k$ \cite{Pinchover-89}.
Let $\abs{b_1(x)-b_2(x)}=0$ outside a compact set.
As earlier, we suppose that $P_1$ is critical and $\lamstr_1=0$.
Assume that $\Tilde V = \mbox{max} \{V_1, V_2 \}$ decays faster than
$(1 + |x|)^{-2 - \epsilon} $ at infinity.
In particular, we may choose $V_2(x)=(1 + |x|)^{-2 - \varepsilon}$.
Then $\Tilde V$ satisfies the estimate above,
and hence, as before, there exists a positive supersolution $\Psi$ to
$\Lg_2 \Psi + \Tilde{V}\Psi\le 0$ 
on the complement of some compact set in $\Rd$.
Thus \cref{ET2.3A} holds. 
\end{example}

%%%%%%%%%%%%%%%%%%%%%%%%%%%%%%%%%%%%%%%%%%%%%%%%%%%%%%%%%%%%%%%%%%%%%%%%%%%%%%%%
\begin{remark}\label{R2.5}
Recall that the criticality of $\Lg_V-\lamstr$ is equivalent
to the strict monotonicity of $\lamstr$ at $V$ on the right by \cref{T2.1}.
However, strict right monotonicity does not necessarily imply strict monotonicity
of $\lamstr$.
Later, in \cref{T2.4}, we show that if $\lamstr$ is strictly
monotone at $V$, then we do not require \cref{ET2.3A}.
Also observe that if $V\in\cB_0(\Rd)$ and $\lamstr$ is not strictly monotone at $V$,
then $\lamstr(V)\le 0$.
Indeed, since $\lamstr$ is not strictly monotone at $V$ and $\lamstr(V)\ge\lamstr(-V^-)$
it is obvious that $\lamstr(V)\le 0$.
In addition, the following hold. 
If $X$ is not positive recurrent and $\lamstr(0)=0$,
then $\lamstr(V)=0$, otherwise
$\lamstr(-V^-)\le \lamstr(V)<0=\lamstr(0)$.
However, this implies that $X$ is geometrically
ergodic \cite[Theorem~2.7]{ABS}, and therefore, positive recurrent.
If $a$ is bounded and uniformly elliptic, and
$\abs{x}^{-1}\langle b(x), x\rangle \to 0$ as $\abs{x}\to \infty$,
then $\lamstr(V)=0$ for any Lipschitz $V\in\cC_0(\Rd)$,
since by \cite[Proposition~6.2]{Ichihara-15} $\lamstr(V)\ge 0$.
Therefore assuming that
$\lamstr_1=0$ in \cref{EG2.5,EG2.6} is not very restrictive.
\end{remark}

In \cite{BCN}, Berestycki, Caffarelli and Nirenberg asked the following question.
Is it true that if there exists a
bounded, sign-changing solution $\Psi$ to $\Delta \Psi+V\Psi=0$ in $\Rd$,
for some locally bounded potential V, then necessarily $\lambda^*(V)> 0$?
This question has been resolved in \cites{Barlow-98, BCN, GG-98}, and the answer
is ``yes" if and only if $d=1,2$.
Applying \cref{T2.1,T2.3} we can extend the sufficiency part of this
answer to a more general class of elliptic operators.
Noting that the Brownian motion is recurrent for $d=1,2$, and transient
for higher dimensions,
we focus on elliptic operators $\Lg$ satisfying
(A1)--(A3) which are generators of a recurrent process.
Using \cref{T2.1,T2.3} we obtain the following two corollaries.

%%%%%%%%%%%%%%%%%%%%%%%%%%%%%%%%%%%%%%%%%%%%%%%%%%%%%%%%%%%%%%%%%%%%%%%%%%%%%%%%
\begin{corollary}\label{C2.1}
Suppose the solution of \cref{E2.1} is recurrent, and $V$ is a locally bounded function
which does not change sign outside some compact set in $\Rd$.
Then the existence of a bounded, sign-changing solution $\Psi\in\Sobl^{2,d}(\Rd)$
to $\Lg\Psi+V\Psi=0$ implies that $\lamstr(V)>0$.
\end{corollary}

\begin{proof}
Since $(\bm1, 0)$ is an eigenpair of $\Lg$ and the corresponding twisted
process is given by $X$,
it follows by \cref{T2.1} that $\Lg$ is a critical operator with
principal eigenvalue $0$. Moreover, $\sC_\Lg=\{c\bm1\,\colon c\in (0, \infty)\}$.
We apply \cref{T2.3}, with $\Lg_1=\Lg_2=\Lg$, $V_1=0$, $\lamstr_1=0$, $V_2=V$,
and $\Psi_1^*=c\bm1$.
Suppose $\lamstr(V)\le 0$.
If $V$ is positive outside a compact set, then \cref{ET2.3A} holds
with $\Tilde\Phi=\Psi^*_2$, the principal eigenfunction corresponding to $\lamstr(V)$.
On the other hand, if $V$ is negative outside a compact set,
then \cref{ET2.3A} holds for some positive $\Tilde\Phi$ by
\cite[Theorem~1.4]{Berestycki-15} (see also \cref{EG2.4}).
It then follows from \cref{T2.3} that $\Psi$ is a ground state,
and therefore cannot be sign-changing.
This contradicts the hypothesis that $\lamstr(V)\le 0$,
and completes the proof.
\end{proof}

%%%%%%%%%%%%%%%%%%%%%%%%%%%%%%%%%%%%%%%%%%%%%%%%%%%%%%%%%%%%%%%%%%%%%%%%%%%%%%%%
\begin{corollary}\label{C2.2}
Let the process \cref{E2.1} be recurrent and $V\le 0$ be a bounded function.
Then there does not exist any nonconstant bounded solution $u$ to
\begin{equation}\label{EC2.2A}
\Lg u + V u \,=\, 0\,.
\end{equation}
\end{corollary}

\begin{proof}
Since $\lamstr(V)\le 0$, \cref{C2.1} implies that
any bounded solution $u$ to \cref{EC2.2A} cannot be sign-changing.
So without loss of generality we assume that $0\le u< C$. Then
$C-u$ is a positive supersolution of $\Lg u=0$.
Since $\Lg$ is critical by hypothesis, 
it has a unique supersolution (up to a multiplicative constant).
Hence $u$ must be constant.
\end{proof}

The conclusion of \cref{C2.2} might not hold if $V\not\le0$. 
For instance, in dimension $d=2$ we know that the standard Wiener process is recurrent.
But $u(x, y)=\sin(x)\,\sin(y)$ satisfies $\Delta u + 2u=0$.
\cref{C2.2} is also comparable to \cite[Theorem~1.7]{Pinchover-07}.
Note that for $V=0$ the operator in \cref{EC2.2A}
is critical in the sense of Pinchover (see \cref{T2.1} above).
Therefore \cref{C2.2} provides a
Liouville property for the perturbed operator.

As shown in \cref{T2.1},
criticality is equivalent to the strict right monotonicity
of the principal eigenvalue $\lamstr$. However, if
we assume strict monotonicity of $\lamstr(\Lg_1, V_1)$ at $V_1$,
then \cref{T2.3} holds for a bigger class of potentials without assuming \cref{ET2.3A}. 
This is the subject of our next result. Also note that the the theorem
which follows
provides sufficient conditions for strict monotonicity of the principal
eigenvalue of the perturbed problem.

%%%%%%%%%%%%%%%%%%%%%%%%%%%%%%%%%%%%%%%%%%%%%%%%%%%%%%%%%%%%%%%%%%%%%%%%%%%%%%%%
\begin{theorem}\label{T2.4}
Let $\Lg_1$ be a small perturbation of $\Lg_2$,
$V_i\in\Lpl^\infty(\Rd)$, $i=1,2$, and $V_1-V_2\in\cB_0(\Rd)$.
Let $\lamstr_i$ denote the principal eigenvalue of $\Lg_i+V_i$, $i=1,2$,
and suppose that $\lamstr_1$ is strictly monotone at $V_1$.
Suppose also that $\lamstr_2\le\lamstr_1$,
and that there exists $\Psi\in\Sobl^{2, d}(\Rd)$, which may be sign-changing
but $\Psi^+\ne 0$, that satisfies
\begin{equation}\label{ET2.4A}
\Lg_2\Psi+V_2\Psi \,\ge\, \lamstr_1\Psi\,,
\end{equation}
and  
\begin{equation}\label{ET2.4B}
\Psi^+(x)\le C\, \Psi^*_1(x)\quad \text{for all\ } x\in\Rd\,,
\end{equation}
for some constant $C>0$.
Then $\lamstr_2$ is strictly monotone at $V_2$, and $\Psi=\Psi^*_2$
\textup{(}up to a multiplicative constant\/\textup{)}, where $\Psi^*_2$ is the principal
eigenfunction of $\Lg_2+V_2$.
\end{theorem}

%%%%%%%%%%%%%%%%%%%%%%%%%%%%%%%%%%%%%%%%%%%%%%%%%%%%%%%%%%%%%%%%%%%%%%%%%%%%%%%%
\begin{remark}\label{Spectral_gap}
Strict monotonicity sometimes implies an interesting spectral property.
To explain this, we restrict ourselves to symmetric operators.
In particular, we consider a second-order elliptic operator in $\cD$ in
divergence form given by
\begin{equation*}
\Lg u\,=\,\mbox{div} \bigl(A\,\nabla u\bigr)\,,
\end{equation*}
where $A\colon\Rd\to\mathcal{S}_+^{d\times d}$ is locally non-degenerate.
The assumptions on the coefficients are the same as before.
Let $\D\nu = \rho(x) \D{x}$, where $\rho(x)$ is a positive measurable
function on $\cD$.
The operator $\Lg$ is self-adjoint in the space $\Lp^2(\cD,\D\nu)$
(in the sense of the Friedrichs extension).

Let $V \in \Lp^{\infty}(\cD)$,
and $\sigma(\Lg_V)$ denote  the $\Lp^2(\cD,\D\nu)$-spectrum of the
Friedrichs extension of $\Lg_V$, which is also denoted as $\Lg_V$, abusing
the notation in the interest of simplicity.
We next show that if $\lamstr(V)$ is
strictly monotone at $V$,
then it must be an isolated eigenvalue in $\sigma(\Lg_V)$.
Indeed, by Persson's formula (see \cite{PER-60} or \cite[Proposition~4.2]{BMY})
the supremum of the essential spectrum $\sigma_{\mathsf{ess}}(\Lg_V)$ is given by 
\begin{equation*}
\lambda_{\infty}(V) \;\df\; \inf \;\{ \lambda\, \colon \exists \;
K \Subset \cD\,,\; \sC_{\Lg_{V} - \lambda}(\cD \setminus K ) \ne\varnothing \}\,.
\end{equation*}
In addition, $\lambda^{*}(V)$ is the supremum of $\sigma(\Lg_V)$.
It is clear that $\lambda_{\infty}(V) \le \lambda^{*}(V)$.
We claim that $\lambda_{\infty}(V) < \lamstr(V)$.
Arguing by contradiction, let us assume $\lambda_{\infty}(V) = \lambda^{*}(V)$. 
It is known that 
\begin{equation}\label{ess_sup}
\sigma_{\mathsf{ess}}(\Lg_V) \; =\; \sigma_{\mathsf{ess}}(\Lg_V - h)
\end{equation}
for any $h \in \cC_0(\cD)$.
Using \cref{ess_sup}, we have $\lambda_{\infty}(V) = \lambda_{\infty}(V - h)$
for all $h \in \cC^+_0(\cD)$.
By hypothesis, $\lambda^*(V)$ is strictly monotone at $V$ and therefore, we have 
\begin{equation*}
\lambda^{*}(V - h) \;<\; \lambda^{*}(V) \;=\; \lambda_{\infty}(V)
\;=\; \lambda_{\infty}(V - h)\;\le\; \lamstr(V-h)\,.
\end{equation*}
Thus we arrive at a contradiction, which implies that
$\lambda_{\infty}(V) < \lambda^{*}(V)$
(for a more general related result see \cite[Theorem~2.5]{ABS}).
Since the Friedrichs extension is a self-adjoint operator,
its spectrum can be written as
$\sigma(\Lg_{V}) = \sigma_{\mathsf{ess}}(\Lg_V) \cup \sigma_{\mathsf{dis}}(\Lg_V)$,
with $\sigma_{\mathsf{ess}}(\Lg_V) \cap \sigma_{\mathsf{dis}}(\Lg_V) = \varnothing$,
where $\sigma_{\mathsf{dis}}(\Lg_V)$ is the discrete spectrum.
On the other hand, \cite[Theorem~1.1]{Beckus-17}
shows that $\lambda^{*}(V)\in\sigma(\Lg_V)$.
Therefore, $\lambda^{*}(V)$ is an isolated eigenvalue in $\sigma(\Lg_{V})$.
\end{remark}

%%%%%%%%%%%%%%%%%%%%%%%%%%%%%%%%%%%%%%%%%%%%%%%%%%%%%%%%%%%%%%%%%%%%%%%%%%%%%%%%%%%%%%%
\begin{remark}
Let $P_i = \Lg+ V_{i}$, $i = 1, 2$, be self-adjoint operators,
and $\lambda_{\infty, i}$ 
denote the supremum of the essential spectrum of $P_i$.
If $V_1-V_2\in\cC_0(\Rd)$, then it is known that
$\sigma_{\mathsf{ess}}(P_1) = \sigma_{\mathsf{ess}}(P_2)$, which in turn implies that
$\lambda_{\infty, 1} = \lambda_{\infty, 2}$.
Suppose that the hypotheses of \cref{T2.4} hold. Then using \cref{T2.4}
and \cref{Spectral_gap} we 
deduce that $\lamstr_2 > \lamstr_{\infty, 2}$, and that the corresponding operator
$P_2-\lamstr_2$ is critical. In particular, \cref{T2.4}
provides a necessary condition
for the spectral gap of the operator $P_2$.
\end{remark}

%%%%%%%%%%%%%%%%%%%%%%%%%%%%%%%%%%%%%%%%%%%%%%%%%%%%%%%%%%%%%%%%%%%%%%%%%%%%%%%%%%%%%%%%
\begin{example}
Let $\cD = \mathbb{R}^d \setminus \{ 0 \}$, where $d \ge 3$,
and consider the Hardy operator 
\begin{equation*}
\Lg_{V} \df \Delta + \frac{(d-2)^2}{4} \frac{1}{|x|^2}\,.
\end{equation*}
Then it is well known (see \cite{BMY}) that for this operator we have
$\lambda^{*}(V) = \lambda_{\infty}(V)$.
Hence $\lambda^{*}(V)$ cannot be strictly monotone at $V$, although
it is strictly right monotone. 
\end{example}

There is a large class of operators for which the strict
monotonicity property holds. 
The following example suggests that the
assumptions in \cref{T2.2,T2.4} hold for a
large class of operators.

%%%%%%%%%%%%%%%%%%%%%%%%%%%%%%%%%%%%%%%%%%%%%%%%%%%%%%%%%%%%%%%%%%%%%%%%%%%%%%%%
\begin{example}
Suppose that the solution of \cref{E2.1} is recurrent. Consider two functions
$\Tilde{V}_i\in\cC^+_0(\Rd)$, $i=1, 2$, which are compactly supported.
Then as shown in \cite[Theorem~2.7]{ABS}, the map
$\beta\mapsto \Lambda^i_\beta\df \lamstr(\beta \Tilde{V}_i)$
is strictly monotone in $[0, \infty)$, and
$\Lambda^i_0=0$, for $i=1,2$.
Since $\beta\mapsto\Lambda^i_\beta$ is an increasing,
convex function \cite{Berestycki-15}, we have
$\lim_{\beta\to\infty}\Lambda^i_\beta=\infty$.
Therefore, for any $\beta_1>0$, we can find $\beta_2>0$ such that
$\Lambda^1_{\beta_1}=\Lambda^2_{\beta_2}$. Thus by defining 
$V_i\df\beta_i\Tilde{V}_i$, $i=1,2$, we note that
$\lamstr_1=\lamstr(V_1)=\lamstr(V_2)=\lamstr_2$, and $V_i$ has compact support.
On the other hand, $\Lg+V_i-\lamstr_i$ is critical by \cref{T2.1}.
In fact, the corresponding
twisted processes are geometrically ergodic by \cite[Theorem~2.7]{ABS}.
Thus, if $\Psi^*_i$, $i=1,2$, are the principal eigenfunctions, then they have a
stochastic representation by \cref{T2.1}.
Hence, if we choose $r$ large enough
such that $\text{support}(V_i)\subset B_r$, we have
\begin{align*}
\Psi^*_i(x)\,=\, \Exp_x\Bigl[e^{-\lamstr_i\uuptau_r} \Psi^*_i(X_{\uuptau_r})
\Ind_{\{\uuptau_r<\infty\}}\Bigr],\quad x\in B^c_r\,.
\end{align*}
Since $\lamstr_1=\lamstr_2$, it is easy to see from the above that
$\Psi^*_2\le C\, \Psi^*_1$ for some $C>0$.
\end{example}

%%%%%%%%%%%%%%%%%%%%%%%%%%%%%%%%%%%%%%%%%%%%%%%%%%%%%%%%%%%%%%%%%%%%%%%%%%%%%%%%
\section{Proofs of \texorpdfstring{\cref{T2.1,T2.2,T2.3,T2.4}}{}}\label{Sec-3}

Before we proceed with the proofs of the results in \cref{Sec-2},
let us recall the It\^o--Krylov formula \cite[p.~122]{Krylov}
for generalized derivatives.
Let $\cD$ be a bounded domain in $\Rd$ with smooth
boundary and $V\in \Lpl^\infty(\Rd)$.
Let $\uptau=\uptau(\cD)$.
Then for any $\varphi\in\Sobl^{2,d}(\Rd)$, we have
\begin{equation}\label{ito-krylov}
\Exp_x\Bigl[e^{\int_0^{T\wedge\uptau} V(X_s)\, \D{s}}\,
\varphi(X_{T\wedge \uptau})\Bigr]-\varphi(x)
\,=\, \Exp_x\biggl[\int_0^{T\wedge\uptau} e^{\int_0^t V(X_s)\, \D{s}}\,
\Lg_V \varphi (X_t)\, \D{t}\biggr]
\end{equation}
for all $x\in\cD$
and $T>0$.
We start with the proof of \cref{T2.1}.

%%%%%%%%%%%%%%%%%%%%%%%%%%%%%%%%%%%%%%%%%%%%%%%%%%%%%%%%%%%%%%%%%%%%%%%%%%%%%%%%
\begin{proof}[Proof of \cref{T2.1}]
The equivalence between (b), (c) and (d) is established in \cite{ABS}.
Since the twisted process corresponding to an eigenpair $(\Psi,\lambda)$ with
$\lambda>\lamstr(V)$ is transient by \cite[Theorem~2.1\,(c)]{ABS},
part (b) together with \cite[Corollary~2.1]{ABS} imply that $\lambda=\lamstr(V)$.
 
Let us show that (b) $\Rightarrow$ (a).
Suppose that $v\in\Sobl^{2,d}(\Rd)$
is a positive function which satisfies $\Lg v + (V-\lambda)v\le 0$ a.e. in $B_{r_1}^c$,
with $r_1>0$.
Recall that $\uptau_R$ denotes the first exit time from the ball $B_R$.
Then by the It\^o--Krylov formula in \cref{ito-krylov} we have 
\begin{align*}
v(x) &\,\ge\, \Exp_x\Bigl[e^{\int_0^{\uuptau_r\wedge\uptau_R\wedge T}
(V(X_s)-\lambda)\, \D{s}}\, v (X_{\uuptau_r\wedge\uptau_R\wedge T})\Bigr]\\[5pt]
&\,\ge\, \Exp_x\Bigl[e^{\int_0^{\uuptau_r}(V(X_s)-\lambda)\, \D{s}}\,
v(X_{\uuptau_r})\,\Ind_{\{\uuptau_r<\uptau_R\wedge T\}}\Bigr]\,,
\quad x\in B^c_r\cap B_R\,,\ r>r_1\,.
\end{align*}
Now letting first $T\to\infty$, and then $R\to\infty$, and using
Fatou's lemma we obtain
\begin{equation*}
v(x) \,\ge\, \Exp_x\Bigl[e^{\int_0^{\uuptau_r}(V(X_s)-\lambda)\, \D{s}}\,
v (X_{\uuptau_r}) \Ind_{\{\uuptau_r<\infty\}}\Bigr], \quad x\in B^c_r\,,
\ r>r_1\,.
\end{equation*}
Hence (a) follows by applying \cref{ET2.1A}.

Next we show that (a) $\Rightarrow$ (b).
By \cref{C3.2}, which appears later in this section, there exists a ball $\sB$,
a constant $\delta\ge0$, and a positive solution
$\Psi^*\in\Sobl^{2,d}(\Rd)$
to $\Lg\Psi^* + (V+\delta\Ind_\sB-\lambda)\Psi^*= 0$,
such that $\lambda=\lamstr(V+\delta\Ind_\sB)$,
and
\begin{equation}\label{PT2.1A}
\Psi^*(x) \,=\, \Exp_x\Bigl[e^{\int_0^{\uuptau_r}(V(X_s)-\lambda)\, \D{s}}\,
\Psi^*(X_{\uuptau_r}) \Ind_{\{\uuptau_r<\infty\}}\Bigr]
\end{equation}
for $B_r\supset\sB$.
Let $\kappa_0\df\inf_{\Rd}\frac{\Psi^*}{\Psi}$.
By the assumption of minimal growth, we have $\kappa_0>0$.
Note that the value $\kappa_0$ is attained. If not, then the 
function $\Phi=\Psi^*-\kappa_0\Psi$ is positive and satisfies
$\Lg\Phi+ (V-\lamstr)\Phi\le 0$.
Consequently, the minimal growth of $\Psi$ would imply that
$\inf_{\Rd}\left(\frac{\Psi^*}{\Psi}-\kappa_0\right)>0$,
thus contradicting the definition of $\kappa_0$.
Therefore, defining $\Phi=\Psi^*-\kappa_0\Psi$, it is easy to see that
$\Lg\Phi - (V+\delta\Ind_{\sB}-\lamstr)^-\Phi\le 0$,
and $\Phi$ attains its minimum value $0$ in $\Rd$,
which implies by the strong maximum principle that $\kappa_0\Psi^*=\Psi$.
This of course implies that $\delta=0$.
Hence $\lambda=\lamstr(V)$, and, in turn, \cref{PT2.1A} implies that
the twisted process corresponding to the eigenpair $(\Psi,\lambda)$ is recurrent.
This completes the proof.
\end{proof}

%%%%%%%%%%%%%%%%%%%%%%%%%%%%%%%%%%%%%%%%%%%%%%%%%%%%%%%%%%%%%%%%%%%%%%%%%%%%%%%%
We continue with the proof of \cref{T2.2}.

%%%%%%%%%%%%%%%%%%%%%%%%%%%%%%%%%%%%%%%%%%%%%%%%%%%%%%%%%%%%%%%%%%%%%%%%%%%%%%%%
\begin{proof}[Proof of \cref{T2.2}]
Let $K\Subset\cD$ be a compact set such that $V_2-V_1\ge 0$ and
$\Lg_1=\Lg_2$ in $K^{c}$.
Since $\lambda_2^* \le \lamstr_1$, using Harnack's inequality, 
it follows that there exists a positive generalized eigenfunction $\Psi^{*}_2$
corresponding to the generalized eigenvalue $\lamstr_2$, i.e.,
\begin{equation*}
P_2\Psi^*_2 \,=\, \lamstr_2\,\Psi^*_2 \quad \text{in\ } \cD\,.
\end{equation*}
Thus we have
\begin{equation}\label{PT2.2A}
\Lg_1\Psi^{*}_2 + (V_1-\lamstr_1)\Psi^{*}_2 \,\le\,
\Lg_2\Psi^{*}_2 + (V_2-\lamstr_2)\Psi^{*}_2\; =\;0\quad\text{in\ } K^c\,.
\end{equation}
By the minimal growth property of $\Psi^{*}_1$ and \cref{PT2.2A},
we can find a positive constant $\kappa$ and a set $\cD_i\in\mathfrak{D}$,
with $\cD_i\supset K$, such that $\kappa\Psi^{*}_1\le \Psi^{*}_2$ for
all $x \in \cD\setminus\cD_i$.
Let 
\begin{equation*}
\Hat{\kappa}\,=\,\sup_{\cD}\, \frac{\Psi}{\Psi^*_2}
\,=\, \sup_{\cD}\, \frac{\Psi^+}{\Psi^*_2}.
\end{equation*}
Then, using \cref{ET2.2B} and the bound $\kappa\Psi^{*}_1\le \Psi^{*}_2$,
we conclude that $\Hat{\kappa} \in (0, \infty)$.
Let us now define 
\begin{equation*}
\Phi(x) \,\df\, \Hat\kappa \Psi^{*}_2(x) - \Psi(x) \quad \ \mbox{in} \; \cD.
\end{equation*}
We claim that there exists $x_0 \in \cD$ such that $\Phi(x_0) = 0$.
If not, then $\Phi(x) > 0$ in $\cD$.
Then in $K^{c}$ we have 
\begin{align*}
\Lg_1\Phi + (V_1-\lamstr_1)\Phi \;&=\; \Lg_2\Phi + (V_1-\lamstr_1)\Phi 
\\[5pt]
& \,\le\, \Lg_2 \Phi + (V_2-\lamstr_1) \Phi 
\\[5pt]
&\,=\, \bigl( \Lg_2 + V_2-\lamstr_1 \bigr)(\Hat\kappa \Psi^{\star}_{2} - \Psi)
\\[5pt]
&\,\le\, \bigl(\Lg_2+V_2 -\lamstr_1\bigr) \Hat\kappa\Psi^*_2 \,=\,
(\lamstr_2-\lamstr_1)\Hat\kappa\, \Psi^*_2 \,\le\, 0\,.
\end{align*}
By the minimality of the growth of
the ground state $\Psi^{*}_1$, there exist a positive constant
$\kappa_1$ and a compact set $K_2\supset K$ such that 
$\kappa_1 \Psi^{*}_{1} \le \Phi$ in $K^{c}_{2}$.
Next, using \cref{ET2.2B}, we obtain
\begin{equation*}
\Hat{\kappa}\Psi^*_2(x) -\Psi(x) \,\ge\, \frac{\kappa_1}{C} \Psi(x)
\; \Rightarrow\; \frac{\Psi(x)}{\Psi^*_2(x)}\,\le\, \frac{\Hat\kappa}
{1+\nicefrac{\kappa_1}{C}}\;<\;\Hat\kappa\qquad\forall\,x\in K^{c}_{2}\,.
\end{equation*}
Thus the value $\Hat\kappa$ is attained for some $x_0 \in K_2$.
This shows that $\Phi(x_0)=0$ at some $x_0\in \cD$.

On the other hand, $\Phi$ is nonnegative, and it satisfies
\begin{equation*}
\Lg_2 \Phi + (V_2-\lamstr_1) \Phi
\,\le\, (\lamstr_2-\lamstr_1) \Hat\kappa\Psi^{*}_2 \,\le\, 0\quad \text{in\ } \cD\,,
\end{equation*}
which in turn implies that
\begin{equation*}
\Lg_2 \Phi - (V_2-\lamstr_1)^-\Phi
\,\le\, -(V_2-\lamstr_1)^+\Phi \,\le\, 0 \quad \text{in\ } \cD\,.
\end{equation*}
Thus by strong maximum principle we must have $\Phi \equiv 0$ in $\cD$.
This shows that $\Hat{\kappa}\Psi^{*}_2=\Psi$, which implies by \cref{ET2.2A}
that $\lamstr_2 = \lamstr_1$. 

To complete the proof it remains to show that $\Psi^{*}_2$
is a ground state of $\Lg_2 + V_2-\lamstr_2$. 
Consider a compact set $\Tilde K$, and let $v\in\Sobl^{2, d}(\Tilde K^{c})$
be a positive supersolution of
$\Lg_2 + V_2-\lamstr_2$, i.e.,
\begin{equation*}
\Lg_2 v + (V_2-\lamstr_2) v\,\le\, 0 \quad \text{in\ } \Tilde K^{c}\,.
\end{equation*}
By hypothesis, we have
\begin{equation*}
\Lg_1 v + (V_1-\lamstr_1) v\,\le\, \Lg_2 v + (V_2-\lamstr_2) v\,\le\, 0
\qquad\text{on\ }K^c \cap \Tilde K^c\,.
\end{equation*}
Since $\Psi^{*}_1$ has minimal growth at infinity, 
we can find a constant $\kappa_2$ and a compact set $\Tilde K_{2}$ satisfying
$\kappa_2 \Psi^{*}_1\le v$ in $\Tilde K_2^{c}$.
Combining this with \cref{ET2.2B} we have
$\frac{\kappa_2}{C}\, \Psi^{*}_2\le v$ in $\Tilde K^{c}_{2}$.
Therefore $\Psi^{*}_2$ also has minimal growth at infinity, and hence is a
ground state.
This completes the proof. 
\end{proof}

%%%%%%%%%%%%%%%%%%%%%%%%%%%%%%%%%%%%%%%%%%%%%%%%%%%%%%%%%%%%%%%%%%%%%%%%%%%%%%%%

As an immediate corollary to \cref{T2.2}, we have the following generalization of the
result in \cite[Corollary~1.8]{Pinchover-07}

%%%%%%%%%%%%%%%%%%%%%%%%%%%%%%%%%%%%%%%%%%%%%%%%%%%%%%%%%%%%%%%%%%%%%%%%%%%%%%%%
\begin{corollary}
Let $P_1$ and $P_2$ be as in \cref{T2.2}.
Suppose that any $\Psi$ which satisfies \cref{ET2.2A,ET2.2B}
cannot be a solution of $(\Lg_2+V_2-\lamstr_1)\Psi=0$ unless
it is sign-changing.
Then $\lamstr_2>\lamstr_1$.
\end{corollary}

%%%%%%%%%%%%%%%%%%%%%%%%%%%%%%%%%%%%%%%%%%%%%%%%%%%%%%%%%%%%%%%%%%%%%%%%%%%%%%%%
To prove \cref{T2.3,T2.4} we need several lemmas which are stated next.

%%%%%%%%%%%%%%%%%%%%%%%%%%%%%%%%%%%%%%%%%%%%%%%%%%%%%%%%%%%%%%%%%%%%%%%%%%%%%%%%
\begin{lemma}\label{L3.1}
Suppose that $\lamstr(V)$
is not strictly right monotone at $V$.
Then for any ball $\sB$ there exists a constant $\delta>0$ such
that $\lamstr(V) = \lamstr(V+\delta\Ind_\sB)$, and
$\lamstr$ is strictly right monotone at $V+\delta\Ind_\sB$.
\end{lemma}

\begin{proof}
Let $F_\alpha(x) \df V(x) - \lamstr(V) -\alpha$, for $\alpha>0$.
It is evident that the Dirichlet eigenvalue of $-\Lg-F_\alpha$ on every ball
$B_n$ is positive.
Thus by Proposition~6.2 and Theorem~6.1 in \cite{Berestycki-94},
for any $n\in\NN$, the Dirichlet problem
\begin{equation}\label{EL3.1A}
\Lg\varphi_{\alpha,n}(x)+ F_\alpha(x)\,\varphi_{\alpha,n}(x)
\,=\, -\Ind_{\sB}(x)\quad\text{a.e.\ }x\in B_n\,,\quad
\varphi_{\alpha,n}=0\text{\ \ on\ }\partial B_n\,,
\end{equation}
has a unique solution $\varphi_{\alpha,n}\in\Sobl^{2,p}(B_n)\cap\Cc(\Bar{B}_n)$,
for any $p\ge1$.
In addition, by the \emph{refined maximum principle} in
\cite[Theorem~1.1]{Berestycki-94} $\varphi_{\alpha,n}$ is nonnegative.
It is clear that $\varphi_{\alpha,n}$ cannot be identically equal to $0$.
Thus if we write \cref{EL3.1A} as
\begin{equation*}
\Lg\varphi_{\alpha,n}- F_\alpha^-\,\varphi_{\alpha,n}
\,=\, - F_\alpha^+\,\varphi_{\alpha,n} - \Ind_\sB\,,
\end{equation*}
it follows by the strong maximum principle
that $\varphi_{\alpha,n}>0$ in $B_n$.
By the It\^o--Krylov formula in \cref{ito-krylov},
since $\varphi_{\alpha,n}=0$ on $\partial B_n$, we obtain from \cref{EL3.1A} that
\begin{multline}\label{EL3.1B}
\varphi_{\alpha,n}(x) \,=\, \Exp_x\Bigl[\E^{\int_0^{T}
F_\alpha(X_s)\,\D{s}}\,\varphi_{\alpha,n}(X_{T})\,\Ind_{\{T\le\uptau_n\}}\Bigr]\\
+\Exp_x\biggl[\int_0^{T\wedge\uptau_n}
\E^{\int_0^{t}F_\alpha(X_s)\,\D{s}}\,\Ind_\sB(X_t)\,\D{t}\biggr]
\end{multline}
for all $(T,x)\in\RR_+\times B_n$.

Now fix $\alpha>0$.
Let $\Psi$ be a positive principal eigenfunction
of $\Lg+V$ constructed canonically from Dirichlet eigensolutions.
We can scale $\Psi$ so that $\Psi\ge 1$ on $\sB$.
Let $\Psi_\alpha = \alpha^{-1}\Psi$.
Then
\begin{equation*}
\Lg \Psi_\alpha(x) + F_\alpha(x)\,\Psi_\alpha(x)
\,\le\, -\Ind_\sB(x)\qquad\text{a.e.\ }x\in\Rd\,.
\end{equation*}
Using the It\^o--Krylov formula and Fatou's lemma, we obtain
\begin{equation}\label{EL3.1C}
\Psi_\alpha(x) \,\ge\, \Exp_x\Bigl[\E^{\int_0^{\mathfrak{t}}
F_\alpha(X_s)\,\D{s}}\,\Psi_\alpha(X_{\mathfrak{t}})\Bigr]
+ \Exp_x\biggl[\int_0^{\mathfrak{t}}
\E^{\int_0^{\mathfrak{t}}F_\alpha(X_s)\,\D{s}}\,\Ind_\sB(X_t)\,\D{t}\biggr]\,,
\end{equation}
for any finite stopping time $\mathfrak{t}$,
and any $\alpha>0$.
Also, $\Psi$ being an eigenfunction, we have
\begin{align}\label{EL3.1D}
\Psi_\alpha(x)&\,\ge\, \Exp_{x}\,\Bigl[\E^{\int_{0}^{T\wedge\uptau_n}
F_0(X_t)\,\D{t}}\,\Psi_\alpha(X_{T})\,\Ind_{\{T\le\uptau_n\}}\Bigr]\\[5pt]
&\,\ge\, \alpha^{-1}\Bigl(\inf_{B_n}\; \Psi\Bigr)\,\Exp_{x}\,
\Bigl[\E^{\int_{0}^{T\wedge\uptau_n}F_0(X_t)\,\D{t}}\,\Ind_{\{T\le\uptau_n\}}\Bigr]\,.
\nonumber
\end{align}
Thus by \cref{EL3.1D} we have
\begin{multline*}
\Exp_x\Bigl[\E^{\int_0^{T}
F_\alpha(X_s)\,\D{s}}\,\varphi_{\alpha,n}(X_{T})
\,\Ind_{\{T\le\uptau_n\}}\Bigr]\\
\,\le\,
\E^{-\alpha T}\,\biggl(\sup_{B_n}\,\varphi_{\alpha,n}\biggr)
\Exp_x\Bigl[\E^{\int_0^{T}F_0(X_s)\,\D{s}}\,\Ind_{\{T\le\uptau_n\}}\Bigr]\,,
\end{multline*}
and the right hand side tends to $0$ as $T\to\infty$.
Taking limits in \cref{EL3.1B} as $T\to\infty$, using monotone convergence
for the second integral, we obtain
\begin{equation*}
\varphi_{\alpha,n}(x)\,=\, 
 \Exp_x\biggl[\int_0^{\uptau_n}
\E^{\int_0^{t}F_\alpha(X_s)\,\D{s}}\,\Ind_\sB(X_t)\,\D{t}\biggr]\,,
\end{equation*}
which implies by \cref{EL3.1C}
that $\varphi_{\alpha,n}\le\Psi_\alpha$ for all $n\in\NN$.
It therefore follows by the a priori estimates that $\{\varphi_{\alpha,n}\}$
is relatively weakly compact in $\Sob^{2,p}(B_n)$, for any $p\ge1$ and $n\in\NN$,
and thus
$\varphi_{\alpha,n}$ converges uniformly on compact sets
along some sequence $n\to\infty$
to a nonnegative
$\Phi_\alpha\in\Sobl^{2,p}(\Rd)$, for any $p\ge1$, which solves
\begin{equation*}
\Lg \Phi_\alpha(x) + F_\alpha(x)\,\Phi_\alpha(x)
\,=\, -\Ind_\sB(x)\qquad\text{a.e.\ }x\in\Rd\,.
\end{equation*}
It is clear by the strong maximum principle that
$\Phi_\alpha>0$.
Since, as we have already shown, $\varphi_{\alpha,n}\le\Psi_\alpha$ for all $n\in\NN$,
it follows that $\Phi_\alpha\le\Psi_\alpha$.
Using \cref{EL3.1C} with $\mathfrak{t}=T$ and a slightly smaller $\alpha$,
then by \cref{EL3.1B} and dominated convergence, we obtain
\begin{equation}\label{EL3.1E}
\Phi_{\alpha}(x)\,=\, \Exp_x\Bigl[\E^{\int_0^{T}
F_\alpha(X_s)\,\D{s}}\,\Phi_{\alpha}(X_{T})\Bigr]
+\Exp_x\biggl[\int_0^{T}
\E^{\int_0^{t}F_\alpha(X_s)\,\D{s}}\,\Ind_\sB(X_t)\,\D{t}\biggr]
\end{equation}
for all $T>0$ and $x\in\Rd$.
Since \cref{EL3.1B} also holds with $T$ replaced by $T\wedge \uuptau_r$,
then again dominating this by \cref{EL3.1C} with $\mathfrak{t}=\uuptau_r$ and
choosing a slightly smaller $\alpha$, we similarly obtain
\begin{multline}\label{EL3.1F}
\Phi_\alpha(x)\,=\,
\Exp_x\Bigl[\E^{\int_{0}^{\uuptau_r}
F_\alpha(X_s)\,\D{t}}\,\Phi_\alpha(X_{\uuptau_r})\Ind_{\{\uuptau_r<\infty\}}\Bigr]\\
+ \Exp_x\biggl[\int_{0}^{\uuptau_r}
\E^{\int_{0}^{t} F_\alpha(X_s)\,\D{s}}\,\Ind_\sB(X_t)\,\D{t}\biggr]
\end{multline}
for all $x\in B_r^c$ and $r>0$.
Using the bound $\Phi_\alpha\le\Psi_\alpha$ we have 
\begin{align*}
\Exp_x\Bigl[\E^{\int_0^{T} F_\alpha(X_s)\,\D{s}}\,\Phi_{\alpha}(X_{T})\Bigr]
&\,\le\,
\E^{-\alpha T}\Bigl[\E^{\int_0^{T} F_0(X_s)\,\D{s}}\,\Psi_\alpha(X_{T})\Bigr]\\[3pt]
&\,\le\,\E^{-\alpha T}\Psi_\alpha(x)\;\xrightarrow[T\to\infty]{}\;0\,.
\end{align*}
Thus by \cref{EL3.1E}, we obtain
\begin{equation*}
\Phi_\alpha(x) \,=\,
\Exp_x\biggl[\int_0^{\infty}\E^{\int_0^{t}F_\alpha(X_s)\,\D{s}}\,
\Ind_\sB(X_t)\,\D{t}\biggr]\,.
\end{equation*}
Since $\lamstr$ is not strictly right monotone at $V$,
the twisted process is transient, and by \cite[Lemma~2.7]{ABS}
we have
\begin{equation*}
\Exp_0\biggl[\int_0^{\infty}\E^{\int_0^{t}F_0(X_s)\,\D{s}}\,
\Ind_\sB(X_t)\,\D{t}\biggr]\;<\;\infty\,.
\end{equation*}
It follows that $\Phi_{\alpha}(0)$ is bounded uniformly over $\alpha\in(0,1)$,
and therefore is uniformly locally bounded
by the superharmonic Harnack inequality \cite{AA-Harnack}.
Thus letting $\alpha\searrow0$, we obtain a positive
$\Phi$ as a limit of $\Phi_\alpha$, which solves
\begin{equation*}
\Lg \Phi(x) + F_0(x)\,\Phi(x)
\,=\, -\Ind_\sB(x)\qquad\text{a.e.\ }x\in\Rd\,.
\end{equation*}
Write this as
\begin{equation*}
\Lg \Phi(x) + \bigl(V(x)+\Phi^{-1}(x)\Ind_\sB(x)\bigr)\,\Phi(x)
\,=\, \lamstr(V)\Phi(x)\qquad\text{a.e.\ }x\in\Rd\,.
\end{equation*}
On the other hand, taking limits in \cref{EL3.1F} as $\alpha\searrow0$,
choosing $r>0$ such that $B_r\supset \sB$, we obtain
\begin{equation*}
\Phi(x)\,=\,
\Exp_x\Bigl[\E^{\int_{0}^{\uuptau_r}
F_0(X_s)\,\D{t}}\,\Phi(X_{\uuptau_r})\Ind_{\{\uuptau_r<\infty\}}\Bigr]\,.
\end{equation*}
This shows that $\Phi$ has a stochastic representation, which implies
that $\lamstr$ is strictly monotone at $V+\Phi^{-1}\Ind_\sB$ on the right.
Then the monotonicity property of
$\delta\mapsto \lamstr(V+\delta\Ind_\sB)$ implies that
$\lim_{\delta\to\infty}\lamstr(V+\delta\Ind_\sB)>\lamstr(V)$.
So we define
$\delta_0=\inf\{\delta>0\,\colon \lamstr(V+\delta\Ind_\sB)>\lamstr(V)\}$,
then $\lamstr$ is strictly monotone at $V+\delta_0\Ind_\sB$
on the right by \cite[Corollary~2.4]{ABS}.
\end{proof}

%%%%%%%%%%%%%%%%%%%%%%%%%%%%%%%%%%%%%%%%%%%%%%%%%%%%%%%%%%%%%%%%%%%%%%%%%%%%%%%%
\begin{corollary}\label{C3.2}
For any $\lambda>\lamstr(V)$ and ball $\sB$, there exists a constant
$\delta$ such that $\lambda=\lamstr(V+\delta\Ind_\sB)$ and $\lamstr$
is strictly right monotone at $V+\delta\Ind_\sB$.
\end{corollary}

\begin{proof}
By Lemma~\ref{L3.1} there exists $\delta^*\ge 0$ such
that $\lamstr(V+\delta^* \Ind_\sB)=\lamstr(V)$,
and $\lamstr$ is strictly right monotone at $V+\delta^* \Ind_\sB$.
Recall that
 the map $\delta\mapsto \lamstr(V+\delta\Ind_\sB)$ is non-decreasing and convex.
Since $\lamstr(V+\delta \Ind_\sB)>\lamstr(V)$ for any $\delta>\delta^*$
by strict right monotonicity, it follows that
this map is strictly increasing and convex on $[\delta^*,\infty)$, and hence
its image is $[\lamstr(V),\infty)$.
\end{proof}

Using \cref{L3.1} we can show the following.

%%%%%%%%%%%%%%%%%%%%%%%%%%%%%%%%%%%%%%%%%%%%%%%%%%%%%%%%%%%%%%%%%%%%%%%%%%%%%%%%
\begin{lemma}\label{L3.2}
Let $\cD=\Rd$.
Assume the hypotheses of \cref{T2.2} and
(A1)--(A2), and in addition,
suppose
that $V_2-V_1\in\cB_0(\Rd)$, and $\Lg_1=\Lg_2$ outside a compact set $K$.
Then $\lamstr_1=\lamstr_2$.
\end{lemma}

\begin{proof}
Suppose that $\lamstr_2<\lamstr_1$.
By \cref{L3.1} there exists $\delta\ge0$, such that
 $\lamstr_2(V_2+\delta\Ind_\sB)$ is strictly monotone at
 $V_2+\delta\Ind_\sB$ on the right,
 and $\lamstr_2(V_2+\delta\Ind_\sB)=\lamstr_2$.
Let $\Phi_\delta$ denote the ground state corresponding
to $\lamstr_2(V_2+\delta\Ind_\sB)$.
Then
\begin{equation*}
\Lg_1\Phi_\delta + (V_1-\lamstr_1)\Phi_\delta\,=\, (\lamstr_2-V_2-\delta\Ind_\sB+ V_1-
\lamstr_1)\Phi_\delta
\end{equation*}
outside the compact set $K$.
Hence by the minimal growth property of $\Psi^*_1$ we have
$\Psi^*_1\le \kappa_1 \Phi_\delta$.
Note that the choice of $\sB$ is arbitrary.
This means we can select $\sB$ so that $\Psi>0$ on $\sB$.
Therefore
\begin{equation*}
\Lg_2\Psi + (V_2+\delta\Ind_\sB)\Psi \,\ge\, \lamstr_1\Psi.
\end{equation*}
Moreover, $\frac{\Psi}{\Phi_\delta}\le \frac{\Psi^+}{\Phi_\delta}$
is bounded above by \cref{ET2.2B}.

By $\Tilde\Lg$ we denote the generator of the twisted process \cref{twisted}
corresponding to $(\Phi_\delta, \lamstr(V_2+\delta\Ind_{\sB}))$ and $\Lg_2$.
Therefore
\begin{equation*}
\Tilde\Lg f \,=\, \Lg_2 f + 2 \langle a_2(x)\grad\varphi_\delta, \grad f\rangle\,,
\quad \text{for}\; f\in\cC^2(\Rd)\,,
\end{equation*}
where $\varphi_\delta=\log\Phi_\delta$.
Since the twisted process \cref{twisted} corresponding
to $(\Phi_\delta, \lamstr(V_2+\delta\Ind_\sB))$ is recurrent by \cref{T2.1},
it exists for all time.
Moreover, we note that
for $\widehat\Phi=\frac{\Psi}{\Phi_\delta}$ we obtain from \cref{ET2.2A} that
\begin{equation*}
\Tilde\Lg\widehat\Phi - (\lamstr_1-\lambda^*_2)\widehat\Phi\,\ge\, 0\,.
\end{equation*}
Now since $\widehat\Phi$ is bounded above,
by applying the It\^o--Krylov formula
to the above equation, we obtain
\begin{equation*}
\widehat\Phi(x)\,\le\, \Tilde\Exp_x\bigl[e^{-(\lamstr_1-\lambda^*_2)T}
\widehat\Phi^+(\Hat{Y}_T)\bigr]
\,\le\, \norm{\widehat\Phi^+}_\infty e^{-(\lamstr_1-\lambda^*_2)T}\,,
\quad\forall\; T>0\,\,.
\end{equation*}
Letting $T\to\infty$ in this inequality, it follows that $\Phi(x)\le 0$ for all $x$,
which contradicts the fact that $\Psi^+\ne 0$.
Hence we have $\lamstr_1=\lamstr_2$.
\end{proof}

%%%%%%%%%%%%%%%%%%%%%%%%%%%%%%%%%%%%%%%%%%%%%%%%%%%%%%%%%%%%%%%%%%%%%%%%%%%%%%%%

Note that the generators $\Lg_1$ and $\Lg_2$ agree outside the compact set $K$.
Therefore, the processes associated to these
generators must agree up to the hitting time $\uuptau(K)$. 

%%%%%%%%%%%%%%%%%%%%%%%%%%%%%%%%%%%%%%%%%%%%%%%%%%%%%%%%%%%%%%%%%%%%%%%%%%%%%%%%
\begin{lemma}\label{L3.3}
Let the assumptions of \cref{T2.3} hold, and
$r>0$ be large enough so that $K\subset B_r$. Then we have
\begin{equation}\label{EL3.3A}
\Psi(x)\,\le\, \Exp_x\Bigl[\E^{\int_0^{\uuptau_r}(V_2(X_s)-\lamstr_1)\, \D{s}}\,
 \Psi^+(X_{\uuptau_r}) \Ind_{\{\uuptau_r<\infty\}}\Bigr]\,.
\end{equation}
\end{lemma}

\begin{proof}
Choose $R>r$ and $x\in B_R\setminus B_r$.
Applying the It\^o--Krylov formula to \cref{ET2.4A} we obtain
\begin{align}\label{EL3.3B}
\Psi(x) &\,\le\, \Exp_x\Bigl[\E^{\int_0^{\uuptau_r}(V_2(X_s)-\lamstr_1)\, \D{s}}\,
\Psi(X_{\uuptau_r}) \Ind_{\{\uuptau_r<\uptau_R\wedge T\}}\Bigr]\\
&\mspace{90mu}
+ \Exp_x\Bigl[\E^{\int_0^{\uptau_R}(V_2(X_s)-\lamstr_1)\, \D{s}}\,
\Psi(X_{\uptau_R}) \Ind_{\{\uptau_R<\uuptau_r\wedge T\}}\Bigr]\nonumber\\
&\mspace{180mu} + \Exp_x\Bigl[\E^{\int_0^{T}(V_2(X_s)-\lamstr_1)\, \D{s}}\,
\Psi(X_{T}) \Ind_{\{T\le \uuptau_r\wedge\uptau_R\}}\Bigr]\nonumber\\[5pt]
&\,\le\, \Exp_x\Bigl[\E^{\int_0^{\uuptau_r}(V_2(X_s)-\lamstr_1)\, \D{s}}\,
\Psi^+(X_{\uuptau_r}) \Ind_{\{\uuptau_r<\uptau_R\wedge T\}}\Bigr]\nonumber\\
&\mspace{90mu}
+  \underbrace{\Exp_x\Bigl[\E^{\int_0^{\uptau_R}(V_2(X_s)-\lamstr_1)\, \D{s}}\,
\Psi^+(X_{\uptau_R}) \Ind_{\{\uptau_R<\uuptau_r\wedge T\}}\Bigr]}_{\sI_1}
\nonumber\\
&\mspace{180mu}+ \underbrace{\Exp_x\Bigl[\E^{\int_0^{T}(V_2(X_s)-\lamstr_1)\, \D{s}}\,
\Psi(X_{T}) \Ind_{\{T\le \uuptau_r\wedge\uptau_R\}}\Bigr]}_{\sI_2}\,.\nonumber
\end{align}

We first show that $\sI_2$ tends $0$ as $T\to \infty$.
By $(\Psi_R, \lambda_R)$ we denote the
principal eigenpair of $\Lg_2 + V_2$ in $B_R$ with Dirichlet boundary condition.
It is known that $\lambda_R$ is strictly increasing to $\lamstr_2$ as 
$R\to\infty$.
An application of the It\^o--Krylov formula shows that
\begin{align}\label{EL3.3C}
\Psi_{R+1}(x) &\,=\, \Exp_x\Bigl[\E^{\int_0^{\uuptau_r}(V_2(X_s)-\lambda_{R+1})\, \D{s}}
\,\Psi_{R+1}(X_{\uuptau_r}) \Ind_{\{\uuptau_r<\uptau_R\wedge T\}}\Bigr]
\\ 
&\mspace{60mu} + \Exp_x\Bigl[\E^{\int_0^{\uptau_R}(V_2(X_s)-\lambda_{R+1})\, \D{s}}\,
\Psi_{R+1}(X_{\uuptau_R}) \Ind_{\{\uptau_R<\uuptau_r\wedge T\}}\Bigr]\nonumber
\\
&\mspace{120mu} + \Exp_x\Bigl[\E^{\int_0^{T}(V_2(X_s)-\lambda_{R+1})\, \D{s}}\,
\Psi_{R+1}(X_{T}) \Ind_{\{T\le \uuptau_r\wedge\uptau_R\}}\Bigr]\nonumber
\end{align}
for $x\in B_R\setminus B_r$.
Since $\lambda_R<\lamstr_2\le \lamstr_1$ and $\Psi_{R+1}>0$ in $B_{R+1}$, we
deduce that
\begin{align*}
\sI_2 &\,=\, \Exp_x\Bigl[\E^{\int_0^{T}(V_2(X_s)-\lamstr_1)\, \D{s}}\,
\Psi(X_{T}) \Ind_{\{T\le \uuptau_r\wedge\uptau_R\}}\Bigr]\\[5pt]
&\,\le\, \frac{1}{\min_{B_R}\Psi_{R+1}}\max_{B_{R}} \abs{\Psi}\,
\Exp_x\Bigl[\E^{\int_0^{T}(V_2(X_s)-\lamstr_1)\, \D{s}}\,\Psi_{R+1}(X_{T})
\Ind_{\{T\le \uuptau_r\wedge\uptau_R\}}\Bigr] \\[5pt]
&\le \; \frac{\E^{(\lambda_{R+1}-\lamstr_1)T}}{\min_{B_R}\Psi_{R+1}}
\Bigl(\max_{B_{R}}\,
\abs{\Psi}\Bigr) \; \Psi_{R+1}(x)\to 0, \quad \text{as\ } T\to\infty,
\end{align*}
where in the last inequality we used \cref{EL3.3C}.

Therefore letting $T\to\infty$ in \cref{EL3.3B} and using the monotone
convergence theorem, we obtain
\begin{align}\label{EL3.3D}
\Psi(x) &\,\le\, \Exp_x\Bigl[\E^{\int_0^{\uuptau_r}(V_2(X_s)-\lamstr_1)\, \D{s}}\,
\Psi^+(X_{\uuptau_r}) \Ind_{\{\uuptau_r<\uptau_R\}}\Bigr] \\[3pt] 
&\mspace{150mu}
+\underbrace{\Exp_x\Bigl[\E^{\int_0^{\uptau_R}(V_2(X_s)-\lamstr_1)\, \D{s}}\,
\Psi^+(X_{\uptau_R}) \Ind_{\{\uptau_R<\uuptau_r\}}\Bigr]}_{\sI_3}\,.\nonumber
\end{align}
We next show that $\limsup\sI_3\le 0$ as $R\to\infty$. 
Recall that $P_1-\lamstr_1$ is critical and therefore, by \cref{T2.1},
we have
\begin{equation}\label{EL3.3E}
\Psi^*_1(x)\;= \; \Exp_x\Bigl[\E^{\int_0^{\uuptau_r}(V_1(X_s)-\lamstr_1)\, \D{s}}\,
 \Psi^*_1(X_{\uuptau_r}) \Ind_{\{\uuptau_r<\infty\}}\Bigr],
\quad x\in B^c_r\,.
\end{equation}
Since $\Psi^+\le C\Psi^*_1$ by \cref{ET2.4B}, we see using \cref{EL3.3E} that 
\begin{align}\label{EL3.3F}
\sI_3 &\;\le \; C \Exp_x\Bigl[\E^{\int_0^{\uptau_R}(V_2(X_s)-\lamstr_1)\, \D{s}}\,
\Psi^*_1(X_{\uptau_R}) \Ind_{\{\uptau_R<\uuptau_r\}}\Bigr]
\\[5pt]
&\,=\,C \Exp_x\biggl[\E^{\int_0^{\uptau_R}(V_2(X_s)-\lamstr_1)\, \D{s}}\,
\Ind_{\{\uptau_R<\uuptau_r\}}\nonumber\\
&\mspace{200mu}\Exp_{X_{\uptau_R}} 
\Bigl[\E^{\int_0^{\uuptau_r}(V_1(X_s)-\lamstr_1)\, \D{s}}\,
\Psi^*_1(X_{\uuptau_r}) \Ind_{\{\uuptau_r<\infty\}}\Bigr]\biggr]\nonumber
\\[5pt]
&\,\le\,C \Exp_x\Bigl[\E^{\int_0^{\uuptau_r}(\Tilde{V}(X_s)-\lamstr_1)\, \D{s}}\,
\Ind_{\{\uptau_R<\uuptau_r<\infty\}} 
\Psi^*_1(X_{\uuptau_r}) \Bigr]\,,\nonumber
\end{align}
where in the third line we used strong Markov property.
On the other hand, using \cref{ET2.3A} we note that
\begin{equation*}
\Exp_x\Bigl[\E^{\int_0^{\uuptau_r}(\Tilde{V}(X_s)-\lamstr_1)\, \D{s}}\,
\Ind_{\{\uuptau_r<\infty\}} \Bigr]\;<\; \infty, \quad \text{for}\; \abs{x}>r\,.
\end{equation*}
Therefore, since $\uptau_R\to \infty$ a.s. as $R\to\infty$,
applying the dominated convergence theorem to \cref{EL3.3F} we have 
\begin{equation}\label{EL3.3G}
\limsup_{R\to\infty}\; \sI_3\,\le\, 0\,.
\end{equation}
Hence, \cref{EL3.3A} follows from \cref{EL3.3D,EL3.3G}
by applying  the monotone convergence theorem.
\end{proof}

We are now ready to present the proofs of \cref{T2.3,T2.4}.

%%%%%%%%%%%%%%%%%%%%%%%%%%%%%%%%%%%%%%%%%%%%%%%%%%%%%%%%%%%%%%%%%%%%%%%%%%%%%%%%
\begin{proof}[Proof of \cref{T2.3}]
Without loss of generality, we may assume that the compact $K$ is large enough so that
there exists a ball $\sB\subset K$ satisfying
$\Psi>0$ in $\sB$. Using \cref{L3.1},
we deduce that there exists $\delta\ge0$ such that
$\lamstr(V_2+\delta\Ind_\sB)=\lamstr_2$, and $\lamstr$ is strictly monotone at
$V_2+\delta\Ind_\sB$ on the right.
Let $\Phi_\delta$ be the ground state of the operator
$\Lg + V_2+\delta\Ind_\sB-\lamstr_2$.
Then, for any $r>0$, we have from \cref{T2.1} that
\begin{equation}\label{PT2.3A}
\Phi_\delta(x)\,=\, \Exp_x\Bigl[e^{\int_0^{\uuptau_r}(V_2(X_s)
+\delta\Ind_{\sB}(X_s)-\lamstr_2)\, \D{s}}\,
\Phi_\delta(X_{\uuptau_r}) \Ind_{\{\uuptau_r<\infty\}}\Bigr],
\quad x\in B^c_r\,.
\end{equation}
Fix $r>0$ large enough so that $K\subset B_r$.
Since $\lamstr_2\le \lamstr_1$ we obtain from \cref{L3.3} and \cref{PT2.3A} that
$\Psi\le \kappa_1 \Phi_\delta$. Define $\widehat\Phi=\frac{\Psi}{\Phi_\delta}$.
Let $\Tilde\Lg$ be the generator of the twisted process
corresponding to $(\Phi_\delta, \lamstr_2)$ and  $\Lg_2$. 
Since $\Lg_2\Psi+(V_2+\delta\Ind_\sB-\lamstr_1)\Psi\ge 0$, we have
\begin{equation}\label{PT2.3B}
\Tilde\Lg \widehat\Phi + (\lamstr_2-\lamstr_1)\widehat\Phi\,\ge\, 0\,.
\end{equation}
Thus repeating the arguments in the proof of \cref{L3.2}, we obtain
$\lamstr_1=\lamstr_2$.
But  it then follows from \cref{PT2.3B} that
$\{\widehat\Phi(Y_s)\}$ is a submartingale which is bounded above.
This of course, implies that $\widehat{\Phi}(Y_s)$ converges almost surely as 
$s\to\infty$.
Since $Y_s$ is recurrent, $\widehat{\Phi}$ has to be constant,
implying that $\Psi=\kappa_2\Phi_\delta$ for some positive $\kappa_2>0$.
Using \cref{ET2.3B}, we obtain $\delta=0$, and this completes the proof.
\end{proof}

%%%%%%%%%%%%%%%%%%%%%%%%%%%%%%%%%%%%%%%%%%%%%%%%%%%%%%%%%%%%%%%%%%%%%%%%%%%%%%%%
\begin{proof}[Proof of \cref{T2.4}]
Let $K$ be a compact set such that $\Lg_1\equiv\Lg_2$ in $K^c$.
Let $h\in\cC^+_0(\Rd)$ be a function with compact support.
Then we know that
\begin{equation*}
\beta\mapsto \Lambda_\beta\,=\,\lamstr(V_{1}+\beta h)
\end{equation*}
is an increasing, convex function \cite[Proposition~2.3]{Berestycki-15}.
In addition, it is strictly monotone at $\beta=0$.
Let $\beta_c\df\inf\,\{\beta\in\RR\,\colon\Lambda_\beta>\Lambda_{-\infty}\}$.
It is then clear that $\beta_c<0$, and hence it follows from
\cite[Theorem~2.7]{ABS} that for some $\beta<0$, close to $0$, 
the twisted process corresponding to the eigenpair
$(\Psi_\beta, \Lambda_\beta)$ and $\Lg_1$ is recurrent
(in fact, geometrically ergodic), and 
$\Lambda_\beta<\lamstr_1=\Lambda_0$.
We also have
\begin{equation}\label{PT2.4A}
\Lg_1 \Psi_\beta + (V_1+\beta h)\Psi_\beta\,=\, \Lambda_\beta \Psi_\beta\,.
\end{equation}
Moreover, by \cref{T2.1}, $\Psi_\beta$ has a stochastic representation, i.e.,
\begin{equation}\label{PT2.4B}
\Psi_\beta(x) \,=\, \Exp_x\Bigl[e^{\int_0^{\uuptau_r}(V_1(X_s)-\Lambda_\beta)\, \D{s}}\,
\Psi_\beta(X_{\uuptau_r}) \Ind_{\{\uuptau_r<\infty\}}\Bigr].
\end{equation}
In \cref{PT2.4B} we use a radius $r$ large enough so that
the support of $h$ and the set $K$ lie in $B_r$.
Also by \cref{L3.2} we have $\lamstr_1=\lamstr_2$.
Let $\delta=\frac{1}{2}(\lamstr_1-\Lambda_\beta)>0$.
It is clear that we can choose $r$ large enough so that
\begin{equation*}
V_2(x)-\lamstr_2 +\delta \,=\, V_2(x)-\lamstr_1 +\delta\; < V_1(x) - \Lambda_\beta
\quad \text{for all\ } \abs{x}\ge r\,.
\end{equation*}
For such a choice of $r$, we note from \cref{PT2.4B} that
\begin{equation}\label{PT2.4C}
\Exp_x\Bigl[e^{\int_0^{\uuptau_r}(V_2(X_s)-\lamstr_2+\delta)\, \D{s}}\,
\Ind_{\{\uuptau_r<\infty\}}\Bigr]\;<\;\infty, \quad \abs{x}\ge r.
\end{equation}
Using \cref{PT2.4C} and the arguments in \cite[Theorem~2.2]{ABS}
(see for instance, (2.30) in \cite{ABS}) it is easy to show
that $\lamstr_2$ is strictly monotone at $V_2$.

Therefore, in order to complete the proof, it remains to show that $\Psi$ is a positive
multiple of $\Psi^*_2$. 
Since $V_1+\beta h-\Lambda_\beta\ge V_1-\lamstr_1$ outside some compact
set $K_0$, we obtain from \cref{PT2.4A} that
\begin{equation*}
\Lg_1 \Psi_\beta + (V_1-\lamstr_1)\Psi_\beta \,\le\, 0\qquad \forall\,x\in K_0^c\,.
\end{equation*}
Therefore, by the minimal growth at infinity of $\Psi^*_1$, we can find
a constant $\kappa_\beta$ satisfying $\Psi^*_1\le\kappa_\beta\Psi_\beta$ in $\Rd$.
Combining this with \cref{ET2.4B}, we have $\Psi^+\le C\kappa_\beta\Psi_\beta$.
As earlier, we fix $r$ large enough so that
$V_2(x)-\lamstr_2< V_1(x)-\Lambda_\beta$, $\Lg_1\equiv\Lg_2$,
and $h(x)=0$ for $\abs{x}\ge r$.
We apply the It\^o--Krylov formula to \cref{ET2.4A} to obtain
\begin{multline*}
\Psi(x)\,\le\, \Exp_x\Bigl[e^{\int_0^{\uuptau_r}(V_2(X_s)-\lamstr_2)\, \D{s}}\,
\Psi(X_{\uuptau_r}) \Ind_{\{\uuptau_r<\uptau_R\}}\Bigr]\\[5pt]
+ \Exp_x\Bigl[e^{\int_0^{\uptau_R}(V_2(X_s)-\lamstr_2)\, \D{s}}\, \Psi(X_{\uptau_R}) 
\Ind_{\{\uuptau_r>\uptau_R\}}\Bigr]\,.
\end{multline*}
By the choice of $r$, we can estimate the second term as follows
\begin{multline}\label{PT2.4D}
\Exp_x\Bigl[e^{\int_0^{\uptau_R}(V_2(X_s)-\lamstr_2)\, \D{s}}\, \Psi^+(X_{\uuptau_r}) 
\Ind_{\{\uuptau_r>\uptau_R\}}\Bigr]\\[5pt]
\,\le\, \kappa_2 \Exp_x\Bigl[e^{\int_0^{\uptau_R}(V_1(X_s)-\Lambda_\beta)\, \D{s}}\,
\Psi_\beta(X_{\uptau_R}) \Ind_{\{\uuptau_r>\uptau_R\}}\Bigr]\,.
\end{multline}
The right hand side of \cref{PT2.4D} tends to $0$, as $R\to\infty$, by \cref{PT2.4B}.
Hence letting $R\to\infty$, we obtain
\begin{equation*}
\Psi(x) \,\le\, \Exp_x\Bigl[e^{\int_0^{\uuptau_r}(V_2(X_s)-\lamstr_2)\, \D{s}}\,
\Psi(X_{\uuptau_r}) \Ind_{\{\uuptau_r<\infty\}}\Bigr]\,.
\end{equation*}
Since $\Psi^*_2$ also has a stochastic representation by \cref{T2.1}, this implies that
$\Psi\le \kappa_1 \Psi^*_2$
for some $\kappa_1>0$.
With $\Phi=\frac{\Psi}{\Psi^*_2}$ we have
\begin{equation*}
\Tilde\Lg\Phi\;\ge \;0\,,
\end{equation*}
where $\Tilde\Lg$ is the generator of twisted process $Y$ corresponding to
$(\Psi^*_2, \lamstr_2)$ and $\Lg_2$.
Thus, $\{\Phi(Y_s)\}$ is a submartingale which is bounded from above.
Since the twisted process $Y$ is recurrent by \cref{T2.1}, $\Phi$ must be constant.
Since $\Psi^+\ne 0$, this implies that
$\Psi$ is a positive function, which means of course, 
that it is a positive multiple of $\Psi^*_2$.
\end{proof}

One interesting by-product of the proof of \cref{T2.4} is the corollary
that follows.
This result however might be known, but we could not locate it
in the literature.

%%%%%%%%%%%%%%%%%%%%%%%%%%%%%%%%%%%%%%%%%%%%%%%%%%%%%%%%%%%%%%%%%%%%%%%%%%%%%%%%
\begin{corollary}
Let $\Lg$ be the operator in \cref{E-Lg}, and 
$\lamstr$ be the principal eigenvalue of $\Lg+V$,
where $V$ is a locally bounded function. 
In addition, suppose that $\Lg+V-\lamstr$ is critical, and
let $\Psi^*_1$ denote the ground state.
Then, there does not exist any non-zero solution $\Psi\in\Sobl^{2, d}(\Rd)$ of
$\Lg \Psi + V\Psi =\lambda \Psi$, for  $\lambda> \lamstr$,
with $\abs{\Psi}\le \kappa \Psi^*_1$.
\end{corollary}

%%%%%%%%%%%%%%%%%%%%%%%%%%%%%%%%%%%%%%%%%%%%%%%%%%%%%%%%%%%%%%%%%%%%%%%%%%%%

We can improve the above results to a larger class of potentials if we impose a 
`stability' condition of the underlying dynamics $X$. Let us assume the
following
\begin{itemize}
\item[\bf(H)] There exists a lower-semicontinuous, inf-compact function
$\ell\colon\Rd\to[0, \infty)$ such that 
\begin{equation*}
\limsup_{T\to\infty}\;\frac{1}{T}\, \Exp_x\Bigl[e^{\int_0^T \ell(X_s)\, \D{s}}\Bigr]
\;<\; \infty \quad \text{for all\ } x\in\Rd\,.
\end{equation*}
\end{itemize}

By $\sorder(\ell)$ we denote the collection of functions $f\colon\Rd\to\RR$ satisfying
\begin{equation*}
\limsup_{\abs{x}\to\infty}\, \frac{\abs{f(x)}}{\ell(x)}\,=\,0\,.
\end{equation*}
We say that the elliptic operator
$\Lg$ satisfies (H) if the process $X$ with extended generator $\Lg$ satisfies (H).
It is easy to see that under hypothesis (H), the process is recurrent.
Therefore, if (H) holds for $\Lg_1$, it follows from \cite[Lemma~2.3]{ari-anup}
that $\lamstr_1(\ell)$ is finite.
Moreover, there exists a positive eigenfunction $\varphi_1\in\Sobl^{2, p}(\Rd)$, $p>1$,
with $\inf_{\Rd}\varphi_1>0$, that satisfies
\begin{equation*}
\Lg_1 \varphi_1 + (\ell-\lamstr_1(\ell))\varphi_1\,=\, 0, \quad \text{in}\ \Rd\,.
\end{equation*}
If $\Lg_2$ is a small perturbation of $\Lg_1$, then $\Lg_2$ also satisfies (H).
To see this, consider a ball $\sB\subset\Rd$
such that $\Lg_1=\Lg_2$ in $\sB^c$.
Let $\chi\colon\Rd\to[0, 1]$ be a smooth function that vanishes in $\sB$ and
equals $1$ outside a ball $B_r\supset\Bar\sB$.
Define $\varphi_2=(1-\chi)+\chi\varphi_1$. 
Note that $\varphi_2=1$ in $\sB$, and
$\varphi_2\ge 1\wedge\inf_{\Rd}\varphi_1>0$ on $\Rd$. 
Then, for some positive constants $\kappa_1$ and $\kappa_2$, we have
\begin{align}\label{E-H}
\Lg_2\varphi_2 &\,=\, \Lg_2 (1-\chi) + \chi \Lg_2\varphi_1 + \varphi_1\Lg_2\chi
+ 2\langle a_2\grad\chi, \grad\varphi_1\rangle
\\[3pt]
&\;= \; \Lg_1 (1-\chi) + \chi \Lg_1\varphi_1 + \varphi_1\Lg_1\chi
+ 2\langle a_1\grad\chi, \grad\varphi_1\rangle
\nonumber\\[3pt]
&\;= \; \Lg_1 (1-\chi) + \chi (\lamstr_1(\ell)-\ell)\varphi_1 + \varphi_1\Lg_1\chi
+ 2\langle a_1\grad\chi, \grad\varphi_1\rangle
\nonumber\\[3pt]
&\,\le\, (\kappa_1-\ell)\,\varphi_1
\nonumber\\[3pt]
&\,\le\, (\kappa_2-\ell)\,\varphi_2 \quad\text{on}\ \Rd\,.\nonumber
\end{align}
In \cref{E-H}, the first inequality arises from the fact that
$\inf_{\Rd}\varphi_1>0$, while in the second inequality we use the
fact that $\varphi_1=\varphi_2$ on $B_r^c$, and $\inf_{\Rd}\varphi_2>0$.
Equation \eqref{E-H} of course implies that 
\begin{equation*}
\limsup_{T\to\infty}\frac{1}{T} \Exp_x\Bigl[e^{\int_0^T \ell(X^2_s)\, \D{s}}\Bigr]
\;<\; \kappa_2 \quad \text{for all\ } x\in\Rd\,,
\end{equation*}
where $X^2$ denotes the diffusion process with generator $\Lg_2$.

We have the following result.

%%%%%%%%%%%%%%%%%%%%%%%%%%%%%%%%%%%%%%%%%%%%%%%%%%%%%%%%%%%%%%%%%%%%%%%%%%%%%%%%
\begin{theorem}\label{T3.1}
Let all the assumptions of \cref{T2.4} hold, 
except we replace $V_1-V_2\in\cB_0(\Rd)$ with $V_i\in\sorder(\ell)$.
Moreover, assume that (H) holds for $\Lg_1$.
Then the conclusion of \cref{T2.4} also holds.
\end{theorem}

\begin{proof}
By \cite[Theorem~3.2]{ABS} we know that $\lamstr$ is strictly monotone at both
$V_1$ and $V_2$. Therefore, in order to complete the proof, we only need to show that
$\lamstr_1=\lamstr_2$ and $\Psi^*_2=\Psi$. 
Since $\ell$ is inf-compact, (H) implies that the processes
$X^i$, $i=1,2$, are recurrent.
Moreover, there exists a positive $\widehat{V}^i\in\Sobl^{2,p}(\Rd)$, $p\ge 1$, such that
\begin{equation}\label{ET3.1A}
\Lg_i \widehat{V}^i + \ell \widehat{V}^i\,=\, \lamstr_i(\ell) \widehat{V}^i
\quad \text{in\ } \Rd, \; i=1,2\,,
\end{equation}
and $\inf_{\Rd}\widehat{V}^i>0$. Let $B_r$ be a ball such that 
\begin{equation}\label{ET3.1Ab}
\babs{V_i(x)-\max\{\lamstr_1(\ell),\lamstr_2(\ell)\}}\,\le\,
 \theta\bigl(\ell(x)-\max\{\lamstr_1(\ell),\lamstr_2(\ell)\}\bigr)
 \quad \forall\,x\in B_r^c\,,
\end{equation}
and $\Lg_1=\Lg_2$, $i=1,2$, on $B_r^c$, for some constant 
$\theta\in (0,1)$.
Recall that $\uuptau_r$ denotes the first hitting time to $B_r$.
Since both processes agree
outside $B_r$, in what follows we use $X$ to denote any one of these processes.
Then applying the It\^o--Krylov formula
to \cref{ET3.1A}, followed by Fatou's lemma, we obtain
\begin{equation}\label{ET3.1B}
\Exp_x\Bigl[e^{\int_0^{\uuptau_r} (\ell(X_s)-\lamstr_i(\ell))\, \D{s}}\,\widehat{V}^i
(X_{\uuptau_r})\Bigr]\,\le\, \widehat{V}^i(x) \quad \text{for}\; x\in B^c_r\,.
\end{equation}
We can choose $B_r$ large enough so that $\Psi^+\ne 0$ in $B_r$.
Let $\sB\Subset B_r$ be such that $\Psi>0$ in $\sB$. 
By \cref{L3.1} we can find $\delta\ge 0$ such that $\lamstr$
is strictly monotone on the right at $V_2+\delta \bm1_\sB$ and
$\lamstr_2=\lamstr(V_2+\delta \bm1_\sB)$. Let
 $\Psi_\delta$ be
the corresponding principal eigenfunction.
By \cref{T2.1} we have
\begin{equation}\label{ET3.1C}
\Exp_x\Bigl[e^{\int_0^{\uuptau_r} (V_2(X_s)-\lamstr_2)\, \D{s}}\,
\Psi_\delta(X_{\uuptau_r})\Bigr]\,=\, \Psi_\delta(x) \quad \text{for}\; x\in B^c_r\,.
\end{equation}
Since $\Lg_1+V_1-\lamstr_1$ is critical by hypothesis, we have 
\begin{equation}\label{ET3.1D}
\Exp_x\Bigl[e^{\int_0^{\uuptau_r} (V_1(X_s)-\lamstr_1)\, \D{s}}\,
\Psi^*_1(X_{\uuptau_r})\Bigr]\,=\, \Psi^*_1(x) \quad \text{for}\; x\in B^c_r\,.
\end{equation}
It follows by \cref{ET3.1Ab,ET3.1B,ET3.1D} that
$\Psi^*_1(x)\le \kappa (\widehat{V}^1(x))^\theta$ in $\Rd$, for some constant $\kappa$.

We claim that 
\begin{equation}\label{ET3.1E}
\Exp_x\Bigl[e^{\int_0^{\uptau_R} (V_2(X_s)-\lamstr_1)\, \D{s}}\,\Psi^*_1(X_{\uptau_R})
\Ind_{\{\uptau_R<\uuptau_r\}}\Bigr]\;\xrightarrow[R\to\infty]{}\; 0\,.
\end{equation}
To prove the claim we define
$\Gamma(R, m)=\{x\in\partial B_r\,\colon \Psi^*_1(x)\ge m\}$ for $m\ge 1$. Then
\begin{align*}
\Exp_x\Bigl[e^{\int_0^{\uptau_R} (V_2(X_s)-\lamstr_1)\, \D{s}}
& \,\Psi^*_1(X_{\uptau_R})\Ind_{\{\uptau_R<\uuptau_r\}}\Bigr]\\[3pt]
&\le\, m\Exp_x\Bigl[e^{\int_0^{\uptau_R} \theta(\ell (X_s)-\lamstr_1(\ell))\, \D{s}}\,
\Ind_{\{\uptau_R<\uuptau_r\}}\Bigr]\\
&\qquad+ \Exp_x\Bigl[e^{\int_0^{\uptau_R} (V_2(X_s)-\lamstr_1)\, \D{s}}\,
 \Psi^*_1(X_{\uptau_R})\Ind_{\{x\in\Gamma(R, m)\}}\Ind_{\{\uptau_R<\uuptau_r\}}\Bigr]
\\[3pt]
&\le\; m\Exp_x\Bigl[e^{\int_0^{\uptau_R} \theta(\ell (X_s)-\lamstr_1(\ell))\, \D{s}}\,
\Ind_{\{\uptau_R<\uuptau_r\}}\Bigr]\\
&\qquad+ \kappa_1 m^{1-\nicefrac{1}{\theta}}\Exp_x\Bigl[e^{\int_0^{\uptau_R}
(\ell(X_s)-\lamstr_1(\ell))\, \D{s}}\,\widehat{V}^1(X_{\uptau_R})
\Ind_{\{\uptau_R<\uuptau_r\}}\Bigr]
\\[3pt]
&\le\; m\Exp_x\Bigl[e^{\int_0^{\uptau_R} \theta(\ell (X_s)-\lamstr_1(\ell))\, \D{s}}\,
\Ind_{\{\uptau_R<\uuptau_r\}}\Bigr]
+ \kappa_1 m^{1-\nicefrac{1}{\theta}} \widehat{V}^1(x)\,.
\end{align*}
Then \cref{ET3.1E} follows by first letting $R\to\infty$, and then $m\to\infty$.

Applying the It\^o--Krylov formula \cref{ito-krylov} to \cref{ET2.4A}
we obtain
\begin{equation}\label{ET3.1F}
\Psi(x)\,\le\, \Exp_x\Bigl[e^{\int_0^{\uuptau_r\wedge\uptau_R\wedge T}
(V_2(X_s)-\lamstr_1)\, \D{s}}\,\Psi(X_{\uuptau_r\wedge\uptau_R\wedge T})\Bigr]\,,
\quad T>0\,.
\end{equation}
Since $\abs{V_2-\lamstr_2}\le \ell-\lamstr_1(\ell)$ in $B_r^c$, and 
\begin{equation*}
\Exp_x\Bigl[e^{\int_0^{\uuptau_r\wedge\uptau_R}
(\ell(X_s)-\lamstr_1(\ell))\, \D{s}}\Bigr]\;<\; \infty\,, \quad r<\abs{x}<R\,,
\end{equation*}
for every fixed $R>r$, we have
\begin{equation*}
\Exp_x\Bigl[e^{\int_0^{T} (V_2(X_s)-\lamstr_1)\, \D{s}}\,\Psi(X_T)
\Ind_{\{T\le \uuptau_r\wedge \uptau_R\}}\Bigr]\;\xrightarrow[T\to\infty]{}\; 0\,.
\end{equation*}
Hence, first letting $T\to\infty$, and then $R\to\infty$ in \cref{ET3.1F},
and using \cref{ET3.1E,ET2.4B}, we obtain
\begin{equation*}
\Psi(x)\,\le\, \Exp_x\Bigl[e^{\int_0^{\uuptau_r} (V_2(X_s)-\lamstr_1)\, \D{s}}\,
\Psi(X_{\uuptau_r})\Bigr]\,.
\end{equation*}
Combining this with \cref{ET3.1C} we have $\Psi\le C_1 \Psi_\delta$.
Now mimicking the arguments in the last part of the proof of \cref{T2.3}, we
obtain $\lamstr_1=\lamstr_2$,
and $\Psi=\Psi_\delta$ with $\delta=0$.
\end{proof}

We next exhibit a family of operators for which (H) holds.

%%%%%%%%%%%%%%%%%%%%%%%%%%%%%%%%%%%%%%%%%%%%%%%%%%%%%%%%%%%%%%%%%%%%%%%%%%%%%%%%
\begin{example}
Let $\delta_1 I\le a(x)\le \delta_2 I$, for $\delta_1, \delta_2>0$ and $x\in\Rd$.
Also $b(x)=b_1(x)+b_2(x)$ where $b_2\in L^\infty(\Rd)$, and
\begin{equation*}
\langle b_1(x), x\rangle \,\le\, -\kappa \abs{x}^{\alpha}
\quad \text{on the complement of a compact set in\ } \Rd\,,
\end{equation*}
for some constant $\kappa>0$, and some $\alpha\in(1, 2]$.
Let $\zeta$ be a positive, twice differentiable function in $\Rd$ such
that $\zeta(x)=\exp(\theta\abs{x}^\alpha)$ for $\abs{x}\ge 1$.
If we choose $\theta\in(0, 1)$ small enough, then
it is routine to check that there exists $R_0>0$ such that
\begin{equation*}
\Lg \zeta(x)\; \le \; -\frac{\kappa\theta}{2}\abs{x}^{2\alpha-2} \zeta(x)
\qquad\text{for\ }\abs{x}\ge R_0\,.
\end{equation*}
The above inequality is known as a (geometric) Foster--Lyapunov stability condition
and $\zeta$ is generally referred to as a Lyapunov function.
Therefore, if we choose a function $\ell$ which coincides with
$\frac{\kappa\theta}{2}\abs{x}^{2\alpha-2}$ outside a compact set,
then using the above inequality and It\^o's formula one can easily verify that (H) holds.
\end{example}

%%%%%%%%%%%%%%%%%%%%%%%%%%%%%%%%%%%%%%%%%%%%%%%%%%%%%%%%%%%%%%%%%%%%%%%%%%%%%%%%
\section{A lower bound on the decay of eigenfunctions}\label{Sec-4}

The main goal of this section is to exhibit a \textit{sharp}
lower bound on the decay of supersolutions, and also to use
this estimate to prove several results for positive solutions.

%%%%%%%%%%%%%%%%%%%%%%%%%%%%%%%%%%%%%%%%%%%%%%%%%%%%%%%%%%%%%%%%%%%%%%%%%%%%%%%%
\begin{lemma}\label{L4.1}
Suppose that there exist
positive constants $M$ and $\eta_0$, and some $\beta\in [0, 2]$ such that
\begin{equation}\label{EL4.1A}
\babs{\langle b(x), x\rangle}\,\le\, M\abs{x}^\beta\,,\quad\text{and}\quad
\langle \xi, a(x)\xi\rangle\,\ge\, \eta_0 \abs{\xi}^2\quad\forall\,\xi\in\Rd
\end{equation}
for all $x$ outside some compact set in $\Rd$.
Let $\alpha\ge\beta$ and $K, \gamma$ be any positive constants satisfying
\begin{equation}\label{EL4.1B}
K\alpha>\,\frac{M}{2\eta_0}+\sqrt{\frac{M^2}{4\eta_0^2}+ \frac{\gamma}{\eta_0}}\,,\quad\text{and}\quad
\lim_{\abs{x}\to \infty}\; \frac{1}{\abs{x}^\alpha}\sum_{i=1}^d a^{ii}(x)\,=\,0\,,
\end{equation}
and define
$\Lyap(x)\df\exp(-K\abs{x}^\alpha)$.
Then there exists $r_0>0$ such that for every $r\ge r_0$ we have
\begin{equation}\label{EL4.1C}
\Exp_x\Bigl[e^{-\gamma \int_0^{\uuptau_r}
\abs{X_s}^{2\alpha-2}\, \D{s}}\,
\Lyap(X_{\uuptau_r})\Ind_{\{\uuptau_r<\infty\}}\Bigr]\,\ge\, \Lyap(x)
\quad \text{for\ } \abs{x}\ge r.
\end{equation}
\end{lemma}

\begin{proof}
By \eqref{EL4.1B} we see that
\begin{equation}\label{RL01}
\eta_0(K\alpha)^2  - M (K\alpha) -\gamma\,>0\,.
\end{equation}
We have
\begin{align*}
\frac{\partial}{\partial x_i} \Lyap(x) &\,=\, - K\alpha \abs{x}^{\alpha-2} x_i \Lyap\\
\frac{\partial^2}{\partial x_i \partial x_j}\Lyap(x)
&\,=\, (K\alpha)^2 \abs{x}^{2\alpha-4} x_ix_j \Lyap(x)
-K\alpha(\alpha-2)\abs{x}^{\alpha-4}x_ix_j \Lyap(x)\\
&\mspace{280mu}-K\alpha
\abs{x}^{\alpha-2}\Lyap(x)\delta_{ij}
\end{align*}
for $1\le i, j\le d$, and $\abs{x}\ge 1$.
This implies that
\begin{align*}
\Lg\Lyap(x) &\,=\, K\alpha \abs{x}^{2\alpha-2}
\biggl(K\alpha - \frac{\alpha-2}{\abs{x}^\alpha}\biggr)
\frac{\Lyap(x)}{\abs{x}^2} \langle x, a x\rangle\\
&\mspace{180mu}
- K\alpha\abs{x}^{\alpha-2}\Lyap(x) \sum_{i=1}^d a^{ii}(x)
+ \langle b(x), \grad\Lyap(x)\rangle\\
&\,\ge\, K\alpha \abs{x}^{2\alpha-2}\Biggl(K\alpha\eta_0-M
- \frac{\eta_0(\alpha-2)}{\abs{x}^\alpha}
- \frac{1}{\abs{x}^\alpha}\sum_{i=1}^d a^{ii}(x)\Biggr)\Lyap(x)\,,
\end{align*}
which combined with \cref{EL4.1B} and \eqref{RL01} shows that there exists
$r_0\ge1$, such that
\begin{equation}\label{EL4.1D}
\Lg\Lyap(x) \,\ge\, \gamma \abs{x}^{2\alpha-2}\Lyap(x)\qquad
\text{for\ } \abs{x}\ge r_0\,.
\end{equation}
Let $R>r\ge r_0$.
Applying the It\^{o}--Krylov formula to \cref{EL4.1D}, we obtain
\begin{align}\label{EL4.1E}
\Lyap(x) &\,\le\, \Exp_x\Bigl[e^{-\gamma \int_0^{\uuptau_r} \abs{X_s}^{2\alpha-2}\,
\D{s}}\, \Lyap(X_{\uuptau_r})\Ind_{\{\uuptau_r<\uptau_R\}}\Bigr]\\[3pt]
&\mspace{180mu}
+ \Exp_x\Bigl[e^{-\gamma \int_0^{\uptau_R} \abs{X_s}^{2\alpha-2}\, \D{s}}\,
\Lyap(X_{\uptau_R})\Ind_{\{\uptau_R<\uuptau_r\}}\Bigr]\nonumber\\[5pt]
&\,\le\, \Exp_x\Bigl[e^{-\gamma \int_0^{\uuptau_r} \abs{X_s}^{2\alpha-2}\, \D{s}}\,
\Lyap(X_{\uuptau_r})\Ind_{\{\uuptau_r<\infty\}}\Bigr] \nonumber\\[3pt]
&\mspace{180mu}
+ \Exp_x\Bigl[e^{-\gamma \int_0^{\uptau_R} \abs{X_s}^{2\alpha-2}\, \D{s}}\,
\Lyap(X_{\uptau_R})\Ind_{\{\uptau_R<\uuptau_r\}}\Bigr]\,.\nonumber
\end{align}
for $r\le \abs{x}\le R$.
On the other hand, 
\begin{equation*}
\Exp_x\Bigl[e^{-\gamma \int_0^{\uptau_R} \abs{X_s}^{2\alpha-2}\, \D{s}}\,
\Lyap(X_{\uptau_R})\Ind_{\{\uptau_R<\uuptau_r\}}\Bigr]
\,\le\, \Exp_x\Bigl[ \Lyap(X_{\uptau_R})\Ind_{\{\uptau_R<\uuptau_r\}}\Bigr]
\,\le\, e^{-K R^\alpha}\to 0\,,
\end{equation*}
as $R\to\infty$.
Thus by letting $R\to\infty$ in \cref{EL4.1E}, we obtain \cref{EL4.1C}.
\end{proof}

The above result should be compared with Carmona \cite{Carmona-78},
Carmona and Simon \cite{Carmona-Simon},
 where a weaker lower bound was obtained for L\'{e}vy processes. 
In these papers, the stationarity and independent
increment property of the underlying process is used,
and also the proof is much more complicated.
For instance, see \cite[Proposition~4.1]{Carmona-78} when $X$ is a Brownian motion.
We next use \cref{L4.1} to provide a quantitative estimate on the decay of
positive supersolutions in the outer domain.

%%%%%%%%%%%%%%%%%%%%%%%%%%%%%%%%%%%%%%%%%%%%%%%%%%%%%%%%%%%%%%%%%%%%%%%%%%%%%%%%
\begin{theorem}\label{T4.1}
Assume \cref{EL4.1A}, and let 
$\gamma$, $\alpha$, and $\Lyap$ be as in \cref{L4.1}.
Let $\sK\subset\Rd$ be a compact set, and suppose
$u\in\Sobl^{2, d}(\sK^c)$ is a nontrivial nonnegative function such that
\begin{equation}\label{ET4.1A}
\Lg u + V u\,=\, 0 \quad \text{in\ } \sK^c\,,
\end{equation} 
where $V$ is locally bounded,
and $V(x)\ge -\gamma \abs{x}^{2\alpha-2}$ for all $\abs{x}$ sufficiently large.
Then there exists a positive constant $C$, not depending on $u$,
provided we fix $u(x_0)=1$ at some $x_0\in\sK^c$, and $r>0$ such that 
\begin{equation}\label{ET4.1B}
u(x) \,\ge\, C\, \Lyap(x) \quad \text{for\ } \abs{x}> r\,.
\end{equation}
\end{theorem}

\begin{proof}
Let $r_0$ be as in \cref{L4.1}.
By the hypotheses of the theorem
we may choose $r>r_0$ and sufficiently large,
so that applying the It\^o--Krylov formula
to \cref{ET4.1A} we have
\begin{align}\label{ET4.1C}
u(x) &\,\ge\, \Exp_x\Bigl[e^{\int_0^{\uuptau_r} V(X_s)\, \D{s}}\,
u(X_{\uuptau_r})\Ind_{\{\uuptau_r<\uptau_R\}}\Bigr]\\[3pt]
&\mspace{200mu}
+ \Exp_x\Bigl[e^{ \int_0^{\uptau_R} V(X_s)\, \D{s}}\, u(X_{\uptau_R})
\Ind_{\{\uptau_R<\uuptau_r\}}\Bigr] \nonumber\\[5pt]
&\,\ge\, \Exp_x\Bigl[e^{-\gamma\int_0^{\uuptau_r} \abs{X_s}^{2\alpha-2}\, \D{s}}\,
u(X_{\uuptau_r})\Ind_{\{\uuptau_r<\uptau_R\}}\Bigr]
\quad \text{for\ } \abs{x}> r\,.\nonumber
\end{align}
By the Harnack inequality we have
$\min_{\abs{z}=r}\,u(z)\ge \kappa$ for some positive constant $\kappa$
which does not depend on $u$.
We let $R\to\infty$ in \cref{ET4.1C} and apply Fatou's lemma to obtain 
\begin{align*}
u(x) &\,\ge\, \Exp_x\Bigl[e^{-\gamma\int_0^{\uuptau_r} \abs{X_s}^{2\alpha-2}\, \D{s}}\,
 u(X_{\uuptau_r})\Ind_{\{\uuptau_r<\infty\}}\Bigr]
\\[3pt]
&\,\ge\, \biggl(\min_{\abs{z}=r}\,u(z)\biggr)\,
\E^{K r^\alpha}\,
\Exp_x\Bigl[e^{-\gamma\int_0^{\uuptau_r} \abs{X_s}^{2\alpha-2}\, \D{s}}\,
\Lyap(X_{\uuptau_r})\Ind_{\{\uuptau_r<\infty\}}\Bigr]
\\[3pt]
& \,\ge\, \kappa\,\E^{K r^\alpha}\, \Lyap(x)
\quad \text{for\ } \abs{x}> r\,,
\end{align*}
by \cref{EL4.1C}. Thus \cref{ET4.1B} follows.
\end{proof}

\begin{remark}
If $u\in\Sobl^{2,d}(\sK^c)$ is a nontrivial nonnegative
supersolution of \cref{ET4.1A} then it is necessarily positive on
$\sK^c$ by the strong maximum principle.
Thus \cref{ET4.1B} is valid for nonnegative
supersolutions; however the constant $C$ depends, in general, on $u$.
\end{remark}

As an immediate corollary to \cref{L4.1,T4.1} we have the following.

%%%%%%%%%%%%%%%%%%%%%%%%%%%%%%%%%%%%%%%%%%%%%%%%%%%%%%%%%%%%%%%%%%%%%%%%%%%%%%%%
\begin{corollary}\label{C4.1}
Let $u\in \Sobl^{2, d}(\Rd)$ be a nontrivial nonnegative solution of
\begin{equation*}
\trace(a\,\grad^2u) + \langle b, \grad u\rangle + V u\,=\, 0
\quad \text{in\ } \sK^c\,.
\end{equation*}
Here, we assume that $\sup_{\sK^c}\abs{b(x)}\le M$, $\sup_{\sK^c}\abs{V(x)}\le \gamma$,
that $a$ is bounded, and
$\langle \xi, a(x)\xi\rangle\ge \eta_0 \abs{\xi}^2$ for all $\xi\in\Rd$.
Then for every $\varepsilon^\prime>0$ there
exist positive constants $C_{\varepsilon^\prime}$ and $R_{\varepsilon^\prime}$ such that
\begin{equation*}
u(x)\,\ge\, C_{\varepsilon^\prime} \, \exp\Biggl(-\biggl(\frac{M}{2\eta_0}+\sqrt{\frac{M^2}{4\eta_0^2}+ \frac{\gamma}{\eta_0}}
+\varepsilon^\prime\biggr)\abs{x}\Biggr)\,,
\quad \abs{x}\ge R_{\varepsilon^\prime}\,.
\end{equation*}
\end{corollary}

\begin{proof}
Let $K=\frac{M}{2\eta_0}+\sqrt{\frac{M^2}{4\eta_0^2}+ \frac{\gamma}{\eta_0}}+\varepsilon^\prime$,
$\alpha=1$, and $\beta=0$.
Then the result follows from \cref{T4.1}.
\end{proof}

Let us now discuss some important aspects of \cref{T4.1,C4.1}.
When $a=I$, $b=0$, and $V$ is the potential function for the
two body problem, a similar lower bound was obtained by Agmon \cite{Agmon-85}.
In the context of \cref{C4.1}, a lower bound
was also obtained by Kenig, Silvestre and Wang \cite[Theorem~1.5]{Kenig-Silvestre-Wang}
for solutions which can be sign-changing; however it is assumed
in \cite{Kenig-Silvestre-Wang} that $V\le 0$,
$a$ is the identity matrix, and $d=2$.
In contrast, \cref{C4.1} does not require these assumptions, but
applies only to nonnegative solutions $u$.
Note that the lower bound in
\cite[Theorem~1.5]{Kenig-Silvestre-Wang} is of the form $\E^{-C R(\log R)^2}$
in the radial direction $R$, whereas
the lower bound in \cref{C4.1} is of the form $\E^{-C R}$, and hence it is tighter.
When $b=0$, this bound is also sharper than the one
conjectured by Kenig in \cite[Question~1]{Kenig-05}.
In fact, this bound is optimal in some sense.
To see this, take $u(x)=e^{-\abs{x}}$ in $\Rd$.
Then $\Delta u + Vu =0$ in $\{\abs{x}> d\}$ where $V(x)=-1+\frac{d-1}{\abs{x}}$.
Since we can take $\varepsilon^\prime$ arbitrarily small, the bound in \cref{C4.1}
is very sharp.

We apply \cref{C4.1} to semi-linear or quasi-linear operators to find
a lower bound on the decay of solutions.

%%%%%%%%%%%%%%%%%%%%%%%%%%%%%%%%%%%%%%%%%%%%%%%%%%%%%%%%%%%%%%%%%%%%%%%%%%%%%%%%
\begin{corollary}
Grant the hypotheses in \cref{L4.1}.
\begin{itemize}
\item[(a)]
Let $f\colon(0, \infty)\to(0, \infty)$ be a continuous function
such that
\begin{equation*}
\limsup_{s\searrow 0}\;\frac{1}{s} f(s)< +\infty\,,
\end{equation*}
and
 $u\in\Sobl^{2, d}(\Rd)$ be a bounded, positive solution of $\Lg u=f(u)$.
Then there exist constants $\gamma>0$ and $C_\gamma$, depending
on $\norm{u}_\infty$, such that
\begin{equation}\label{EC4.2A}
u(x)\,\ge\, C_\gamma e^{-\gamma\abs{x}}\quad \text{for all\ } \abs{x}\; \text{
sufficiently large}.
\end{equation}

\item[(b)]
Let $\Act_1, \Act_2$ be two compact metric spaces, and
$V, b\colon\Rd\times\Act_1\times\Act_2\to \Rd$ be two continuous functions
with $\norm{V}_\infty<\infty$, and $b(\cdot, v_1, v_2)$ satisfying
\textup{(}A2\textup{)} uniformly
in $(v_1, v_2)\in \Act_1\times\Act_2$.
If $u\in \Sobl^{2, d}(\Rd)$ is a positive solution of 
\begin{equation*}
\min_{v_1\in \Act_1}\;\max_{v_2\in\Act_2}\;\bigl[ a^{ij}\partial_{ij}u
+ b^i(x, v_1, v_2) \partial_i u + V(x, v_1, v_2) u\bigr]\,=\,0\,,
\end{equation*}
then it satisfies \cref{EC4.2A}.
\end{itemize}
\end{corollary}

\begin{proof}
For part (a), note that since $u$ is bounded, we have $f(u)\le C u$ for some constant
$C$ (depending on $\norm{u}_\infty$). Thus we obtain
\begin{equation*}
\Lg u \,\le\, C\, u\,,
\end{equation*}
and the proof follows from \cref{C4.1}.
For part (b), observe that we can find measurable selectors $v^*_i\colon\Rd\to \Act_i$
satisfying 
\begin{equation*}
a^{ij}\partial_{ij}u + b^i\bigl(x, v^*_1(x), v^*_2(x)\bigr) \partial_i u + 
V\bigl(x, v^*_1(x), v^*_2(x)\bigr) u\,=\,0\,.
\end{equation*}
The rest follows as before using \cref{C4.1}.
\end{proof}

In the rest of this section we discuss some connections of \cref{T4.1,C4.1}
with the Landis conjecture, and provide a partial answer to this conjecture.
In 1960s, E. M. Landis conjectured (see \cite{Kondratev-Landis})
that if $u$ is a solution to $\Delta u+ Vu=0$, with $\norm{V}_\infty\le q^2$,
and there exist positive constants $\varepsilon$ and $C_\varepsilon$
such that $\abs{u(x)}\le C_\varepsilon\, \E^{-(q+\varepsilon)\abs{x}}$,
then $u\equiv 0$. He also proposed a weaker 
version of this conjecture which states that if 
$\abs{u(x)}\le C_k \E^{-k\abs{x}}$
for any positive $k$, and some constant $C_k$, then $u\equiv0$. 
This conjecture was disproved by Meshkov in \cite{Meshkov} who constructed
a non-zero solution to $\Delta u + V u=0$ which satisfies
$\abs{u(x)}\le C\E^{-c\abs{x}^{\nicefrac{4}{3}}}$ for some
positive constants $c$ and $C$.
It is also shown in \cite{Meshkov} that if for any $k>0$, there exists a constant
$C_k$ satisfying $\abs{u(x)}\le C_k\E^{-k\abs{x}^{\nicefrac{4}{3}}}$,
then $u$ is identically $0$.
The counterexample by Meshkov has $V$ and $u$ complex valued.
Therefore the Landis conjecture remains open for real valued solutions and potentials.
It is interesting
to note that the Landis conjecture concerns the \emph{unique continuation property}
of $u$ at infinity. 
In practice, Carleman type estimates are commonly used to treat such problems,
but since a Carleman estimate 
does not distinguish between real and complex valued functions, it is hard to
improve the results of Meshkov using such estimates.
Landis' conjecture was recently revisited
by Kenig, Silvestre and Wang \cite{Kenig-Silvestre-Wang}, and
Davey, Kenig and Wang \cite{Davey-Kenig-Wang} for $V\le 0$ and $d=2$.
Note that if $V\le 0$ then Landis' conjecture follows
from the strong maximum principle.
The key contribution of \cites{Kenig-Silvestre-Wang, Davey-Kenig-Wang} is a lower
bound on the decay of $u$.
On the other hand, if we assume $u$ to be nonnegative, then the 
Landis conjecture follows from \cref{C4.1}.
In \cref{T4.2} below we show that Landis' conjecture holds under the assumption that
$\lamstr(V)\le 0$. It should be observed that $\lamstr(V)\le 0$
does not necessarily imply that $V\le 0$.

%%%%%%%%%%%%%%%%%%%%%%%%%%%%%%%%%%%%%%%%%%%%%%%%%%%%%%%%%%%%%%%%%%%%%%%%%%%%%%%%
\begin{theorem}\label{T4.2}
Let $\Lg u+Vu=0$ and suppose that the following hold.
\begin{itemize}
\item[(i)] $a$ is bounded and uniformly elliptic with ellipticity constant $\eta_0$,
$\norm{b}_\infty\le M$, $\norm{V}_\infty\le\gamma$, and $\lamstr(V)\le 0$.
\item[(ii)]
For some positive constants $\varepsilon$ and $C_\varepsilon$, we have 
\begin{equation*}
\abs{u(x)}\,\le\, C_\varepsilon\, \exp\Biggl(-\biggl(\frac{M}{2\eta_0}+\sqrt{\frac{M^2}{4\eta_0^2}+ \frac{\gamma}{\eta_0}}
+\varepsilon \biggr)\abs{x}\Biggr)\,, \quad \forall\; x\in\Rd\,.
\end{equation*}
\end{itemize}
Then $u\equiv 0$.
\end{theorem}

\begin{proof}
Let $\Psi$ be a positive function in $\Sobl^{2, d}(\Rd)$ satisfying
\begin{equation}\label{ET4.2A}
\Lg\Psi + V\Psi \,=\, 0\,.
\end{equation}
Existence of such $\Psi$ follows, for example, from \cite[Theorem~1.4]{Berestycki-15}
and the fact that $\lamstr(V)\le 0$.
By $\Tilde\Lg$ we denote the twisted process corresponding to the eigenpair
$(\Psi, 0)$ i.e.,
\begin{equation*}
\Tilde\Lg f\,=\, \Lg f + 2\langle a\grad\psi, \grad f\rangle\,,
\quad f\in\cC^2(\Rd)\,,
\end{equation*}
where $\psi=\log\Psi$.
Let $\Phi=\frac{u}{\Psi}$. Then it is easy to check from \cref{ET4.2A} that
\begin{equation}\label{ET4.2B}
\Tilde\Lg \Phi =0\,.
\end{equation}
On the other hand, by \cref{C4.1} we have
\begin{equation*}
\Psi(x)\,\ge\, C_{\varepsilon^\prime}\, \exp\Biggl(-\biggl(\frac{M}{2\eta_0}+\sqrt{\frac{M^2}{4\eta_0^2}+ \frac{\gamma}{\eta_0}}
+\varepsilon^\prime \biggr)\abs{x}\Biggr) \quad \forall\;
\abs{x}>R_{\varepsilon^\prime}\,.
\end{equation*}
for $\varepsilon^\prime<\varepsilon$, and constants $C_{\varepsilon^\prime}$
and $R_{\varepsilon^\prime}$.
This of course, implies that $\Phi(x)\to 0$ as $\abs{x}\to\infty$.
Therefore, applying the strong maximum principle to \cref{ET4.2B},
we deduce that $\Phi\equiv 0$, which in turn implies that $u\equiv 0$.
This completes 
the proof.
\end{proof}

%%%%%%%%%%%%%%%%%%%%%%%%%%%%%%%%%%%%%%%%%%%%%%%%%%%%%%%%%%%%%%%%%%%%%%%%%%%%%%%%
\begin{remark}
Recall that $W$ denotes the standard Brownian motion.
Let $V\in \cC_0(\Rd)$ be such that
\begin{equation}\label{ER4.1A}
\Exp_x\Bigl[\E^{\int_0^\infty \frac{1}{2}V^+(W_s)\, \D{s}}\Bigr]\;<\; \infty\,,
\quad \forall\; x\in\Rd\,.
\end{equation}
It is then known that $v(x)\df \Exp_x\Bigl[\E^{\int_0^\infty \frac{1}{2} V^+(W_s)\,
\D{s}}\Bigr]$ is a positive solution to
\begin{equation*}
\Delta v+ V^+ v\,=\,0\,.
\end{equation*}
This of course implies that $\Delta v+ V v\le 0$, and therefore,
$\lamstr(V)\le 0$. For $d\ge 3$, if we have 
\begin{equation*}
\sup_{x\in\Rd} \frac{2}{(d-2)\omega_d}
\int_{\Rd}\frac{V^+(y)}{\abs{x-y}^{d-2}}\;<\; 1\,,
\end{equation*}
with $\omega_d$ denoting the surface measure of the unit sphere in $\Rd$,
then $V$ also satisfies \cref{ER4.1A} by Khasminskii's lemma
\cite[Lemma~B.1.2]{Barry-82}.

Recall that, as shown in \cref{R2.5}, if
$V\in\cC_0(\Rd)$ and $\lamstr$ is not strictly monotone at $V$,
then $\lamstr(V)\le 0$.
This allows us to apply \cref{T4.2}, to conclude that Landis' conjecture holds if
$\lamstr(\Lg, V)$ is not strictly monotone at $V\in\cC_0(\Rd)$.
This of course applies to the general class of operators
$\Lg$ satisfying (A1)--(A3).
\end{remark}

Though we have not been able to prove the Landis conjecture in its full generality,
we can validate this conjecture for a large class of potentials, including
compactly supported potentials.
This can be done with the help of the
Radon transformation and a support theorem from Helgason \cite{Helgason}.
See also \cite{Meshkov-14} which uses a similar approach, albeit for
homogeneous elliptic equations.

%%%%%%%%%%%%%%%%%%%%%%%%%%%%%%%%%%%%%%%%%%%%%%%%%%%%%%%%%%%%%%%%%%%%%%%%%%%%%%%%
\begin{theorem}\label{T4.3}
Suppose $\Delta u + \langle b, \nabla u\rangle +Vu=0$ in $\sB^c$, where $\sB$
is a bounded ball,
and the following hold.
\begin{itemize}
\item[(i)] $\norm{b}_\infty\le M, \, \norm{V}_\infty\le \gamma$.
There exists a ball $B_r$, $r>0$, with $\sB\subset B_r$,
such that $b$ and $V$ are constant in $B^c_r$.
\item[(ii)] For some positive $\varepsilon, C_\varepsilon$, we have 
\begin{equation*}
\abs{u(x)}\,\le\, C_\varepsilon\, \exp\Biggl(-\biggl( \frac{M}{2}+\sqrt{\frac{M^2}{4}+ \gamma}+\varepsilon \biggr)\abs{x}\Biggr)\,, \quad \forall\; x\in\sB^c\,.
\end{equation*}
\end{itemize}
Then $u\equiv 0$.
\end{theorem}

\begin{proof}
Without loss of generality we may assume $0\in\sB$, and
$b(x)=b_0$, $V(x)=k$ for all $x\in B^c_r$.
Also by standard regularity theory of elliptic PDE we may assume that $u$ is
smooth in $\sB^c_1$. Let $(\omega, p)\in S^d\times \RR$
where $S^d$ is the $(d-1)$-dimensional unit sphere in $\Rd$.
We note that any hyperplane $\Rd$ can be identified by $(\omega, p)$ up to the
equality $(-\omega, -p)=(\omega, p)$.
Let $\xi$ be a hyperplane in $\Rd$, i.e., for some $(\omega, p)$ we
have $\xi=\{x\in\Rd\colon \langle x, \omega\rangle = p\}$.
The Radon transformation of $u$ is defined as
\begin{equation*}
\Breve{u}(\xi) \,\df\, \int_{\xi} u(y)\, S(\D{y})\,, \quad \text{where $S(\D{y})$
is the surface measure on}\; \xi\,.
\end{equation*}
We claim that if the hyperplane does not intersect $\sB_1$, then we have 
\begin{equation}\label{ET4.3A}
\Breve{u}(\xi)\,=\, 0\,.
\end{equation}
If \cref{ET4.3A} is true, then since $u$ decays exponentially fast to $0$ at infinity
by (ii), then the support theorem
\cite[Theorem~1.2.6 and Corollary~1.2.8]{Helgason} implies that $u\equiv 0$ in $B^c_r$.
This in turn, implies that $u\equiv 0$ by the unique continuation property of
H\"{o}rmander \cite[Theorem~2.4]{Hormander}.

In order to complete the proof we need to prove \cref{ET4.3A}.
Note that if we define $v(x)=u(\mathscr{M}x)$, where $\mathscr{M}$ is any
rotation matrix, then
\begin{equation*}
\Delta v + \langle \mathscr{M}^{\transp} b_0, \grad v\rangle + kv\,=\, 0\,,
\quad x\in \sB^c_1\,.
\end{equation*}
Also a rotation does not change the norm of $b$.
Therefore, without loss of generality, we may assume that
$\xi=\xi(\kappa_0)\df \{x\in\Rd\; \colon\; x_1=\kappa_0, \, \kappa_0>0\}$.
Define for $s\ge \kappa_0$,
\begin{equation*}
w(s) \,\df\, \int_{\xi(s)} u(y) \, S(\D{y}) = \int_{\RR^{d-1}} u(s, \Bar{x})\,
\D{\Bar{x}}, \quad \text{where}\; \Bar{x}=(x_2, \ldots, x_d)\in\RR^{d-1}\,.
\end{equation*}
Note that $w$ is smooth in $[\kappa_0, \infty)$ due to the smoothness of $u$
and its decay at infinity.
Moreover,
\begin{align}\label{ET4.3B}
\frac{\D^2{w}(s)}{\D{s^2}} &\,=\, \int_{\RR^{d-1}} \frac{\D^2{u}}{\D{s^2}}(s, \Bar{x})\,
\D{\Bar{x}} \\[3pt]
&\,=\, - \int_{\RR^{d-1}} \sum_{i=2}^d \partial_{ii}u(s, \Bar{x})\, \D{\Bar{x}}
- \int_{\RR^{d-1}} \sum_{i=1}^d b^i_0 \partial_{i}u(s, \Bar{x})\, \D{\Bar{x}}
\nonumber \\
&\mspace{300mu}- k\int_{\RR^{d-1}} u(s, \Bar{x})\, \D{\Bar{x}}\nonumber \\[3pt]
&\,=\, - \int_{\RR^{d-1}} b^1_0 \partial_{1}u(s, \Bar{x})\, \D{\Bar{x}}
- k\int_{\RR^{d-1}} u(s, \Bar{x})\, \D{\Bar{x}}\nonumber \\[5pt]
&\,=\, -b^1_0 \frac{\D{w}(s)}{\D{s}} - kw(s)\,,\nonumber
\end{align}
where in the second equality we use the equation satisfied by $u$,
and in the third equality we use the fundamental theorem of calculus.
Thus we obtain from 
\cref{ET4.3B} a second-order ODE with constant coefficients, given by
\begin{equation}\label{ET4.3C}
\frac{\D^2{w}(s)}{\D{s^2}} + b^1_0 \frac{\D{w}(s)}{\D{s}} + kw(s)\,=\, 0
\quad \text{in\ } [\kappa_0, \infty)\,.
\end{equation}
We solve this ODE explicitly, and using the decay property of $u$ in (ii)
we show that $w(s)=0$ in $[\kappa_0, \infty)$.
In particular, $w(\kappa_0)=0$ which proves \cref{ET4.3A}.
Denote by $\kappa_1=\frac{M}{2}+\sqrt{\frac{M^2}{4}+ \gamma}$. We first show that for $\varepsilon^\prime<\varepsilon$
there exists a positive constant $C_{\varepsilon^\prime}$ such that
\begin{equation}\label{ET4.3D}
\abs{w(s)}\;\le \; C_{\varepsilon^\prime} \, \E^{-(\kappa_1 + \varepsilon^\prime)s}\,,
\quad \text{for}\; s\in[\kappa_0, \infty)\,.
\end{equation}
By (ii) and a choice of $s_0$, satisfying $\sqrt{s_0}(s_0-1)>2$, we get
\begin{align*}
\abs{w(s)}&\;\le \; \int_{\RR^{d-1}} \abs{u(s, \Bar{x})}\, \D{\Bar{x}}
\\[5pt]
&\le \; C_\varepsilon \int_{\RR^{d-1}} \E^{-\kappa_\varepsilon\abs{x} }\, \D{\Bar{x}}
\quad [\mbox{for}\; \kappa_\varepsilon=\kappa_1+\varepsilon]
\\[5pt]
&= \; C_\varepsilon s^{d-1} \int_{\RR^{d-1}} \E^{-\kappa_\varepsilon s
\sqrt{1 + \abs{\Bar{x}}^2} }\, \D{\Bar{x}}
\\[5pt]
&= \; \kappa_2 s^{d-1} \int_0^\infty \E^{-\kappa_\varepsilon s
\sqrt{1 + r^2} }\, r^{d-2} \D{r}
\\[5pt]
&= \; \kappa_2 s^{d-1} \biggl[\int_0^{s_0} \E^{-\kappa_\varepsilon s
\sqrt{1 + r^2} }\, r^{d-2} \D{r} 
+ \int_{s_0}^\infty \E^{-\kappa_\varepsilon s \sqrt{1 + r^2} }\, r^{d-2} \D{r} \biggr]
\\[5pt]
&\le \; \kappa_2 s^{d-1} \biggl[s_0^{d-1} \E^{-\kappa_\varepsilon s}
+ \int_{s_0}^\infty \E^{-\kappa_\varepsilon s (1 + \sqrt{r}) }\, r^{d-2} \D{r} \biggr]
\\[5pt]
&\le \; \kappa_3 s^{d-1} \E^{-\kappa_\varepsilon s}\biggl[ 1+ \int_{s_0}^\infty
\E^{-\kappa_\varepsilon \kappa_0 \sqrt{r} }\, r^{d-2} \D{r}\biggr]\,,
\end{align*}
where in the third inequality we have used the fact that
$\sqrt{1+r^2}> 1+\sqrt{r}$ for $r\ge s_0$.
\Cref{ET4.3D} easily follows from the above estimate.
To solve \cref{ET4.3C} we find the roots of the characteristic polynomial of
the ODE which are
given by
\begin{equation*}
r_1=-\frac{1}{2}b^1_0 + \frac{1}{2}\sqrt{(b^1_0)^2 -4 k}\,, \quad \text{and}
\quad r_2=-\frac{1}{2}b^1_0 - \frac{1}{2}\sqrt{(b^1_0)^2 -4 k}\,.
\end{equation*}
The solution of \cref{ET4.3C} can be written as
\begin{equation*}
w(s)\,=\, c_1 \E^{r_1 s} + c_2 \E^{r_2 s}\,,
\end{equation*}
where the constants $c_1$ and $c_2$ are uniquely determined.
Now if $(b^1_0)^2 -4 k<0$, then the roots are complex and the decay of $w$
is of order $\E^{-\frac{1}{2}b^1_0 s}$ which is larger than 
the RHS of \cref{ET4.3D}.
Therefore, we must have $c_1=c_2=0$.
On the other hand, if $(b^1_0)^2 -4 k\ge 0$, and since
$\babs{-\frac{1}{2}b^1_0 \pm \frac{1}{2}\sqrt{(b^1_0)^2 -4 k}}\le \kappa_1$,
we conclude that $c_1=c_2=0$.
Hence we must have $w(s)=0$ in $[\kappa_0, \infty)$.
This completes the proof.
\end{proof}

As a concluding remark, we show that the Landis conjecture is true for
solutions that satisfy a reverse Poincar\'e inequality.
Specifically, consider an operator $\Lg +V$ with bounded coefficients and $a$
the identity matrix.
We say a solution $u$ to $\Lg u + Vu=0$ satisfies (G) if the following holds:
\begin{itemize}
\item[\bf(G)] There exist positive constants
 $r$ and $C$, independent of $x$, such that 
\begin{equation*}
\int_{B_r(x)}\abs{\grad u(y)}^2\,\D{y}
\,\le\, C \int_{B_r(x)}\abs{u(y)}^2\,\D{y},
\quad \forall\;\abs{x}\gg 1\,.
\end{equation*}
\end{itemize}

%%%%%%%%%%%%%%%%%%%%%%%%%%%%%%%%%%%%%%%%%%%%%%%%%%%%%%%%%%%%%%%%%%%%%%%%%%%%%%%%
\begin{remark}
Note that if $\int_{B_r(x)}\abs{\grad u(y)}^2\,\D{y}=0$ for some $x\in\Rd$, then it is
shown by H\"{o}rmander \cite[Theorem~2.4]{Hormander} that $u\equiv 0$.
\end{remark}

The following (weaker) Landis' conjecture is true for solutions satisfying (G). 
Suppose that $\Delta u + \langle b, \nabla u\rangle +Vu=0$ in $\sB^c$,
with $\sB$ a bounded ball, $u\in \cC^2(\sB^c)$, 
and the following hold.
\begin{itemize}
\item $\norm{b}_\infty\le M, \, \norm{V}_\infty\le q^2$, and $u$ satisfies (G).
\end{itemize}
Then for $\kappa=\bigl[2(M\sqrt{C} + q^2 + C)\bigr]^{\nicefrac{1}{2}}+1$, we have
\begin{equation}\label{ET4.4A}
\int_{B_r(x)} u^2(y)\, \D{y}\;\ge \; C_{\kappa} \exp(-\kappa\, \abs{x})
\quad \text{for all}\; \abs{x}\gg 1\,.
\end{equation}
In particular, if for every $k\in\NN$, we have $\abs{u(x)}\le C_k \E^{-k\abs{x}}$
for some constant $C_k$, then $u\equiv 0$.

In order to prove this claim, we define
\begin{equation*}
v(x)\,\df\,\int_{B_r(x)} u^2(y)\, \D{y}\,.
\end{equation*}
Since $u\in\cC^2(\sB^c)$, we have $v\in\cC^2(\sB^c_1)$ for some large ball
$\sB_1\Supset \sB$.
A straightforward calculation shows that
\begin{align*}
\Delta v(x) &\,=\, 2\int_{B_r(x)} u\Delta u\, \D{y} + 2\int_{B_r(x)}
\abs{\grad u}^2\,\D{y}
\\[5pt]
&\,=\, -2\int_{B_r(x)} u \langle b, \grad{u}\rangle \, \D{y}
-2\int_{B_r(x)} V u^2 \, \D{y} + 2\int_{B_r(x)} \abs{\grad u}^2\,\D{y}
\\[5pt]
&\,\le\, 2M \biggl(\int_{B_r(x)} \abs{u}^2\, \D{y} \biggr)^{\nicefrac{1}{2}}
\biggl(\int_{B_r(x)} \abs{\grad{u}}^2\, \D{y} \biggr)^{\nicefrac{1}{2}}
 + 2q^2 \int_{B_r(x)}u^2 \, \D{y}\\
&\mspace{360mu}+ 2\int_{B_r(x)} \abs{\grad u}^2\,\D{y}
\\[5pt]
&\;\le \; (2M\sqrt{C} + 2q^2 + 2C) v(x)\,.
\end{align*}
Therefore \eqref{ET4.4A} follows from \cref{C4.1}. 

%%%%%%%%%%%%%%%%%%%%%%%%%%%%%%%%%%%%%%%%%%%%%%%%%%%%%%%%%%%%%%%%%%%%%%%%%%%%%%%%
\section*{Acknowledgements}
We wish to convey our sincere thanks to Yehuda Pinchover for his helpful
comments and remarks throughout this project.
The research of Ari Arapostathis was supported in part by the Army Research Office
through grant W911NF-17-1-001, in part by the National Science Foundation through grant
DMS-1715210, and in part by Office of Naval Research through grant N00014-16-1-2956.
The research of Anup Biswas was supported in part by an INSPIRE faculty
fellowship (IFA13/MA-32) and DST-SERB grant EMR/2016/004810.
Debdip Ganguly was supported in part at the Technion by a fellowship of the Israel
Council for Higher Education and the Israel Science Foundation (Grant No. 970/15)
founded by the Israel Academy of Sciences and Humanities.

%%%%%%%%%%%%%%%%%%%%%%%%%%%%%%%%%%%%%%%%%%%%%%%%%%%%%%%%%%%%%%%%%%%%%%%%%%%%%%%%

%\bibliography{AMS-ABG}
%\end{document}

\end{document}